\documentclass{article}
\usepackage{amsmath}
\usepackage{amsthm}
\usepackage{amssymb}
\usepackage[ansinew]{inputenc}
\usepackage{url}
\usepackage[all,dvips,arc,curve,color,frame]{xy}
\input xy
\xyoption{all}

\newtheorem{lemma}{Lemma}[section]

\newtheorem{theorem}[lemma]{Theorem}
\newtheorem{cor}[lemma]{Corollary}
\newtheorem{prob}[lemma]{Problem}
\newtheorem{prop}[lemma]{Proposition}
\newtheorem{defi}[lemma]{Definition}

\newtheorem{nota}[lemma]{Notation}

\DeclareMathOperator\dom{dom}
\DeclareMathOperator\ima{im}
\DeclareMathOperator\rank{rank}
\DeclareMathOperator\spa{span}

\DeclareMathOperator\sym{Sym}

\DeclareMathOperator\aut{Aut}

\newcommand{\miff}{\Leftrightarrow}

\newcommand{\al}{\alpha}
\newcommand{\bt}{\beta}
\newcommand{\als}{\alpha^*}
\newcommand{\bts}{\beta^*}

\newcommand{\gam}{\gamma}
\newcommand{\del}{\delta}
\newcommand{\sig}{\sigma}

\newcommand{\lam}{\lambda}

\newcommand{\vep}{\varepsilon}

\newcommand{\cg}{\mathcal{G}}
\newcommand{\ta}{\mbox{\tiny $A$}}
\newcommand{\tb}{\mbox{\tiny $B$}}
\newcommand{\tc}{\mbox{\tiny $C$}}

\newcommand{\td}{\mbox{\tiny $D$}}

\newcommand{\tk}{\mbox{\tiny $K$}}
\newcommand{\tl}{\mbox{\tiny $L$}}

\newcommand{\hag}{h_{\gam}^{\al}}

\newcommand\mi{\mathcal{I}}
\newcommand\jo{\sqcup}
\newcommand\cgam{\Gamma}
\newcommand\come{\Omega}
\newcommand\nc{\left\lfloor\frac{n}{2}\right\rfloor}
\newcommand\nca{\left\lfloor\frac{n-1}{2}\right\rfloor}

\title{The Commuting Graph of the Symmetric Inverse Semigroup}

\author{Jo\~ao Ara\'ujo\\
  {\small Universidade Aberta, R. Escola Polit\'{e}cnica, 147}\\
  {\small 1269-001 Lisboa, Portugal}\\{\footnotesize \&}\\
  {\small Centro de \'{A}lgebra, Universidade de Lisboa}\\
  {\small 1649-003 Lisboa, Portugal, jaraujo@ptmat.fc.ul.pt}\\\\
 Wolfram Bentz\\
{\small Centro de \'{A}lgebra, Universidade de Lisboa}\\
  {\small 1649-003 Lisboa, Portugal, wolfbentz@googlemail.com}\\ \\
\and Janusz Konieczny\\
{\small Department of Mathematics, University of Mary Washington}\\ 
{\small Fredericksburg, Virginia 22401, USA, jkoniecz@umw.edu}}
\date{}
\begin{document}
\maketitle

\begin{abstract}
The commuting graph of a finite non-commutative semigroup $S$, denoted $\cg(S)$, is a simple graph whose
vertices are the non-central elements of $S$ and two distinct vertices $x,y$ are adjacent if $xy=yx$.
Let $\mi(X)$ be the symmetric inverse semigroup of partial 
injective transformations on a finite set $X$.
The semigroup $\mi(X)$ has the symmetric group $\sym(X)$ of permutations on~$X$ as its group of units.
In 1989, Burns and Goldsmith determined the clique number of the commuting graph of $\sym(X)$. In 2008,
Iranmanesh and Jafarzadeh found an upper bound of the diameter of $\cg(\sym(X))$, and in 2011, Dol\u{z}an and Oblak
claimed that this upper bound is in fact the exact value.

The goal of this paper is to begin the study of the commuting graph of the symmetric inverse semigroup $\mi(X)$. 
We calculate the clique number of $\cg(\mi(X))$, 
the diameters of the commuting graphs of the proper ideals of $\mi(X)$, and the diameter of $\cg(\mi(X))$ 
when $|X|$ is even or a power of an odd prime. We show that when $|X|$ is odd and divisible by at least two primes,
then the diameter of $\cg(\mi(X))$ is either $4$ or $5$.
In the process, we obtain several results about semigroups, 
such as a description of all commutative subsemigroups of $\mi(X)$ of 
maximum order, and  analogous results for 
commutative inverse and commutative nilpotent subsemigroups of $\mi(X)$. 
The paper closes with a number of problems for experts in combinatorics and in group or semigroup theory.

\vskip 2mm

\noindent\emph{$2010$ Mathematics Subject Classification\/}. 05C25, 20M20, 20M14, 20M18.

\vskip 2mm


\end{abstract}

\section{Introduction}
\setcounter{equation}{0}

The commuting graph of a finite non-abelian group $G$ is a simple graph whose
vertices are all non-central elements of $G$ and two distinct vertices $x,y$ are adjacent if $xy=yx$.
Commuting graphs of various groups have been studied in terms of their properties (such as connectivity or diameter),
for example in \cite{BaBu03,Bu06,IrJa08,Se01}.
They have also been used as a tool to prove group theoretic results, for example in \cite{Be83,RaSe01,RaSe02}. 

For the particular case of the commuting graph of the finite symmetric group $\sym(X)$, 
it has been proved \cite{IrJa08} that its diameter is $\infty$ when $|X|$ or $|X|-1$ is a prime, and is at most $5$ 
otherwise. It has been claimed \cite{DoOb11} that if neither $|X|$ nor $|X|-1$ is a prime,
then the diameter of $\cg(\sym(X))$ is exactly~$5$. The claim is correct but the proof
contains a gap (see the end of Section 6). 
The clique number of $\cg(\sym(X))$
follows from the classification of the maximum order abelian subgroups of $\sym(X)$
\cite{BuGo89,KoPr89}. 
In addition, there is a very interesting conjecture (which is still open,
as far as we know)  that there exists a common upper bound of the diameters of the (connected) commuting graphs of finite groups. 

The concept of the commuting graph carries over to semigroups. Suppose $S$ is a finite non-commutative semigroup
with center $Z(S)=\{a\in S:ab=ba\mbox{ for all $b\in S$}\}$. The \emph{commuting graph}
of $S$, denoted $\cg(S)$, is the simple graph (that is, an undirected graph with no multiple edges or loops)
whose vertices are the elements of $S-Z(S)$ and whose edges are the sets $\{a,b\}$ such
that $a$ and $b$ are distinct vertices with $ab=ba$.

In 2011, Kinyon and the first and third author \cite{ArKiKo11} initiated the study of 
the commuting graphs of (non-group) semigroups. They calculated the diameters of the ideals of the 
semigroup $T(X)$ of full transformations on a finite set $X$ \cite[Theorems~2.17 and~2.22]{ArKiKo11},
and for every natural number $n$, constructed a semigroup of diameter $n$  \cite[Theorem~4.1]{ArKiKo11}.
(The latter result shows that the aforementioned conjecture on the diameters of finite groups does not hold for semigroups.)
Finally, the study of the commuting graphs of semigroups led to the solution of a longstanding open problem in semigroup theory
\cite[Proposition~5.3]{ArKiKo11}.

The goal of this paper is to extend to the finite symmetric inverse semigroups 
part of the research already carried out for the finite symmetric groups. 
The \emph{symmetric inverse semigroup} $\mi(X)$ on a set $X$ is the semigroup whose elements
are the partial injective transformations on $X$ (one-to-one functions whose domain and image are included in $X$)
and whose multiplication is the composition of functions. We will write functions on the right ($xf$ rather than $f(x))$)
and compose from left to right ($x(fg)$ rather than $f(g(x))$.
The semigroup $\mi(X)$ is universal for the class
of inverse semigroups since every inverse semigroup can be embedded in some $\mi(X)$ \cite[Theorem~5.1.7]{Ho95}. 
This is analogous to the fact that every group can be embedded in some symmetric
group $\sym(X)$ of permutations on~$X$. We note that $\mi(X)$ contains an identity (the transformation that fixes every element of $X$)
and a zero (the transformation whose domain and image are empty). The class of inverse semigroups is arguably the second most important class of semigroups, after groups, 
because inverse semigroups have applications and provide motivation in 
other areas of study, for example, differential geometry and physics \cite{La98,Pa99}.   

Various subsemigroups of the finite symmetric inverse semigroup $\mi(X)$ have been studied.
One line of research in this area has been the determination of subsemigroups of $\mi(X)$
of a given type that are either maximal (with respect to inclusion) or largest (with respect to order).
(See, for example, \cite{AnFeMi07,GaKo94,Ya99,Ya05}.)

In 1989, Burns and Goldsmith \cite{BuGo89} obtained a complete classification of
the abelian subgroups of maximum order of the symmetric group $\sym(X)$, where 
$X$ is a finite set. These abelian subgroups are of three different types depending on the value
of $n$ modulo $3$, where $n=|X|$. We extend this result to the commutative subsemigroups of $\mi(X)$
of maximum order (Theorem~\ref{tgen}). We also determine the maximum order commutative inverse subsemigroups of $\mi(X)$ (Theorem~\ref{tinv})
and the maximum order commutative nilpotent subsemigroups of $\mi(X)$ (Theorem~\ref{tnil}).
As a corollary of Theorem~\ref{tgen}, we obtain the clique number of the commuting graph of $\mi(X)$ (Corollary~\ref{ccli}).

We also find the diameters of the commuting graphs of the proper ideals of $\mi(X)$ (Theorem~\ref{tpro}),
the diameter of $\cg(\mi(X))$ when $n=|X|$ is even (Theorem~\ref{tdie}) and when $n$ is a power of an odd prime
(Theorem~\ref{tpow}),
and establish that the diameter of $\cg(\mi(X))$ is $4$ or $5$ when $n$ is odd and divisible
by at least two distinct primes (Proposition~\ref{pdio}). The diameter results extend to $\cg(\mi(X))$
the results obtained for $\cg(\sym(X))$ by Iranmanesh and Jafarzadeh \cite{IrJa08}
and  Dol\u{z}an and Oblak \cite{DoOb11}. (However, see our discussion at the end of Section~\ref{scdg} regarding a problem
with Dol\u{z}an and Oblak's proof.)
We conclude the paper with some problems that we believe will be of interest
for mathematicians working in combinatorics and semigroup or group theory (Section~\ref{spro}).

The concept of the commuting graph of a transformation semigroup is central for associative algebras
since, in a sense, the study of associativity is the study of commuting transformations
and centralizers \cite{ArKo12}.
This paper builds upon the results on 
centralizers of transformations in general and of partial injective transformations
in particular \cite{AnArKo11,ArKo03,ArKo04,Ko99,Ko02,Ko03,Ko04,Ko10,KoLi98,Li96}.

Throughout this paper, we fix a finite set $X$ and reserve $n$ to denote the cardinality of $X$.
To simplify the language, we will sometimes say ``semigroup in $\mi(X)$'' to mean ``subsemigroup of $\mi(X)$.''
We will denote the identity in $\mi(X)$ by $1$ and the zero in $\mi(X)$ by $0$.

\section{Commuting Elements of $\mi(X)$}
\setcounter{equation}{0}
In this section, we collect some results about commuting transformations in $\mi(X)$
that will be needed in the subsequent sections.

Let $S$ be a semigroup with zero. An element $a\in S$ is called a \emph{nilpotent} if $a^p=0$
for some positive integer $p$; the smallest such $p$ is called the \emph{index} of $a$.
We say that $S$ is a \emph{nilpotent semigroup} if every element
of $S$ is a nilpotent. A special type of a nilpotent semigroup is a \emph{null semigroup}
in which $ab=0$ for all $a,b\in S$. Note that every nonzero nilpotent in a null semigroup has index $2$.
We say that $S$ is a \emph{null monoid}
if it contains an identity $1$ and $ab=0$ for all $a,b\in S$ such that $a,b\ne1$.
Clearly, all null semigroups and all null monoids are commutative.

For $\al\in \mi(X)$, we denote by $\dom(\al)$ and $\ima(\al)$ the domain and image of $\al$, respectively.
The $\emph{rank}$ of $\al$ is the cardinality of $\ima(\al)$ (which is the same as the cardinality of $\dom(\al)$
since $\al$ is injective). The union $\spa(\al)=\dom(\al)\cup\ima(\al)$ will be called the \emph{span} of $\al$.

Let $\al,\bt\in\mi(X)$. We say that $\bt$ is \emph{contained} in $\al$ (or $\al$ \emph{contains} $\bt$)
if $\dom(\bt)\subseteq\dom(\al)$ and $x\bt=x\al$ for all $x\in\dom(\bt)$.
We say that $\al$ and $\bt$ in $\mi(X)$ are \emph{completely disjoint} if $\spa(\al)\cap\spa(\bt)=\emptyset$.
Let $M=\{\gam_1,\ldots,\gam_k\}$ be a set of pairwise completely disjoint elements of $\mi(X)$.
The \emph{join} of the elements of $M$,
denoted $\gam_1\jo\cdots\jo\gam_k$, is the element $\al$ of $\mi(X)$ whose domain is $\dom(\gam_1)\cup\ldots\cup\dom(\gam_k)$
and whose values are defined by $x\al=x\gam_i$, where $\gam_i$ is the (unique) element of $M$ such that $x\in\dom(\gam_i)$. If $M=\emptyset$,
we define the join to be $0$. 
Let $x_0,x_1,\ldots,x_k$ be pairwise distinct elements of $X$. 
\begin{itemize}
  \item A \emph{cycle} of length $k$ ($k\geq1$), written $(x_0\,x_1\ldots\, x_{k-1})$, 
is an element $\rho\in\mi(X)$ with $\dom(\rho)=\{x_0,x_1,\ldots,x_{k-1}\}$, $x_i\rho=x_{i+1}$ for all $0\leq i<k-1$,
and $x_{k-1}\rho=x_0$.
  \item A \emph{chain} of length $k$ ($k\geq1$), written $[x_0\,x_1\ldots\, x_k]$, 
is an element $\tau\in\mi(X)$ with $\dom(\tau)=\{x_0,x_1,\ldots,x_{k-1}\}$ and $x_i\tau=x_{i+1}$ for all $0\leq i\leq k-1$.
\end{itemize}

The following decomposition result is given in \cite[Theorem~3.2]{Li96}.
\begin{prop}\label{pdec}
Let $\al\in\mi(X)$ with $\al\ne0$. Then there exist unique sets $\cgam=\{\rho_1,\ldots,\rho_k\}$ of cycles 
and $\come=\{\tau_1,\ldots,\tau_m\}$ of chains such that
the transformations in $\cgam\cup\come$ are pairwise completely disjoint and
$\al=\rho_1\jo\cdots\jo\rho_k\jo\tau_1\jo\cdots\jo\tau_m$.
\end{prop}

Let $\al=\rho_1\jo\cdots\jo\rho_k\jo\tau_1\jo\cdots\jo\tau_m$ as in Proposition~\ref{pdec}. 
Note that every $\rho_i$ and every $\tau_j$ is contained in $\al$.
Moreover, for every integer $p>0$, $\al^p=\rho_1^p\jo\cdots\jo\rho_k^p\jo\tau_1^p\jo\cdots\jo\tau_m^p$. 
For example, if 
\[
\al=\begin{pmatrix}1&2&3&4&5&6&7&8\\2&3&4&1&6&7&8&-\end{pmatrix}=(1\,2\,3\,4)\jo[5\,6\,7\,8]\in\mi(\{1,2,\ldots,8\}), 
\]
then
$\al^2=(1\,3)\jo(2\,4)\jo[5\,7]\jo[6\,8]$, $\al^3=(1\,4\,3\,2)\jo[5\,8]$, and $\al^4=(1)\jo(2)\jo(3)\jo(4)$.

Let $\al\in\mi(X)$. Then:
\begin{itemize}
  \item  $\al\in\sym(X)$ if and only if $\al=\rho_1\jo\cdots\jo\rho_k$
is a join of cycles and $\cup_{i=1}^k \dom(\rho_i)=X$. The join $\al=\rho_1\jo\cdots\jo\rho_k$ is equivalent to the cycle decomposition
of $\al$ in group theory. Note that a cycle $(x_0\,x_1\ldots\, x_{t-1})$ differs from the corresponding cycle in $\sym(X)$
in that the former is undefined for every $x\in X-\{x_0,x_1,\ldots,x_{t-1}\}$, while the latter
fixes  every such $x$.
  \item $\al$ is a nilpotent if and only if $\al=\tau_1\jo\cdots\jo\tau_m$ is a join of chains;
and $\al^2=0$ if and only if $\al=[x_1\,y_1]\jo\cdots\jo[x_m,y_m]$ is a join of chains of length $1$, where
we agree that $\al=0$ if $m=0$.
\end{itemize}

The following proposition has been proved in \cite[Theorem~10.1]{Li96}.

\begin{prop}\label{pcen}
Let $\al,\bt\in\mi(X)$. Then $\al\bt=\bt\al$ if and only if the following conditions are satisfied:
\begin{itemize}
  \item [\rm(1)] If $\rho=(x_0\,x_1\ldots\, x_{k-1})$ is a cycle in $\al$ such that some $x_i\in\dom(\bt)$,
then every $x_j\in\dom(\bt)$ and there exists a cycle $\rho'=(y_0\,y_1\ldots\,y_{k-1})$
in $\al$ (of the same length as $\rho$) such that 
\[
x_0\bt=y_j,\,\,x_1\bt=y_{j+1},\ldots,x_{k-1}\bt=y_{j+k-1},
\]
where $j\in\{0,1,\ldots,k-1\}$ and the subscripts on the $y_i$s are calculated modulo $k$;
  \item [\rm(2)] If $\tau=[x_0\,x_1\ldots\, x_k]$ is a chain in $\al$ such that some $x_i\in\dom(\bt)$,
then there are $p\in\{0,1,\ldots,k\}$ and a chain $\tau'=[y_0\,y_1\ldots\, y_m]$ in $\al$, with $m\geq p$,
such that $\dom(\bt)\cap\{x_0,x_1,\ldots,x_k\}=\{x_0,x_1,\ldots,x_p\}$ and
\[
x_0\bt=y_{m-p},\,\,x_1\bt=y_{m-p+1},\ldots,x_p\bt=y_m;
\]
  \item [\rm(3)] If $x\not\in\spa(\al)$ and $x\in\dom(\bt)$, then either $x\bt\not\in\spa(\al)$
or there exists a chain $\tau'=[y_0\,y_1\ldots\, y_m]$ in $\al$ such that $x\bt=y_m$.
\end{itemize}
\end{prop}
The way to remember Proposition~\ref{pcen} is that $\al\bt=\bt\al$ if and only if $\bt$ maps cycles in $\al$
onto cycles in $\al$ of the same length, and it maps initial segments of chains in $\al$ onto terminal segments
of chains in $\al$.

An element $\vep\in\mi(X)$ is an idempotent ($\vep\vep=\vep$) if and only if $\vep=(x_1)\jo(x_2)\jo\cdots\jo(x_k)$ is a join of cycles of length $1$;
and $\sig\in\mi(X)$ is a permutation on $X$ if and only if $\dom(\sig)=X$ and $\sig$ is a join of cycles.
For a function $f:A\to B$ and $A_0\subseteq A$, we denote by $f|_{A_0}$ the restriction of $f$ to $A_0$.

The following lemma will be important in our inductive arguments in Sections~\ref{sinv} and~\ref{sgen}.

\begin{lemma}\label{lres1}
Suppose $\gam\in\mi(X)$ is either an idempotent such that $\gam\notin\{0,1\}$, or a permutation on $X$
such that not all cycles in $\gam$ have the same length. 
Then there is a partition $\{A,B\}$ of $X$
such that $\bt|_A\in\mi(A)$ and $\bt|_B\in\mi(B)$ for all $\bt\in\mi(X)$ such that $\gam\bt=\bt\gam$.
\end{lemma}
\begin{proof} 
Suppose $\gam=(x_1)\jo(x_2)\jo\cdots\jo(x_k)\in\mi(X)$ is an idempotent such that $\gam\notin\{0,1\}$.
Let $A=\dom(\gam)=\{x_1,x_2,\ldots,x_k\}$ and $B=X-A$. Then $A\ne\emptyset$ (since $\gam\ne0$), $B\ne\emptyset$ (since $\gam\ne1$),
and $A\cap B=\emptyset$. Note that $B=X-\spa(\gam)$. Let $\bt\in\mi(X)$ be such that $\gam\bt=\bt\gam$. Let $x_i\in A$ and $y\in B$
be such that $x_i,y\in\dom(\bt)$. 
Then $x_i\bt=x_j\in A$ by (1) of Proposition~\ref{pcen},
and $y\bt\in B$ by (3) of Proposition~\ref{pcen}. Hence $\bt|_A\in\mi(A)$ and $\bt|_B\in\mi(B)$.

Suppose $\gam\in\mi(X)$ is a permutation on $X$ such that not all cycles in $\gam$ have the same length. 
Select any cycle $\rho$ in $\gam$ and let $k$ be the length of $\rho$.
Let 
\[
A=\{x\in X:\mbox{$x\in\spa(\rho')$ for some cycle $\rho'$ in $\gam$ of length $k$}\}
\]
and let $B=X-A$. 
Then $A\ne\emptyset$ (since $\rho$ is a cycle in $\gam$ of length $k$), $B\ne\emptyset$ (since not all cycles
in $\gam$ have length $k$),
and $A\cap B=\emptyset$. Let $\bt\in\mi(X)$ be such that $\gam\bt=\bt\gam$. Let $x\in A$ and $y\in B$
be such that $x,y\in\dom(\bt)$. 
Then $x\bt\in A$ and $y\bt\in B$ by (1) of Proposition~\ref{pcen}.
Hence $\bt|_A\in\mi(A)$ and $\bt|_B\in\mi(B)$.
\end{proof}

It is straightforward to prove the following lemma.

\begin{lemma}\label{lres2}
Let $\{A,B\}$ be a partition of $X$. Suppose $\al,\bt\in\mi(X)$ are such that  
$\al|_A,\bt|_A\in\mi(A)$ and $\al|_B,\bt|_B\in\mi(B)$. Then:
\begin{itemize}
  \item [\rm(1)] $(\al\bt)|_A=(\al|_A)(\bt|_A)$ and $(\al\bt)|_B=(\al|_B)(\bt|_B)$.
  \item [\rm(2)] $\al\bt=\bt\al$ if and only if $(\al|_A)(\bt|_A)=(\bt|_A)(\al|_A)$ and $(\al|_B)(\bt|_B)=(\bt|_B)(\al|_B)$.
\end{itemize}
\end{lemma}

We conclude this section with a lemma that is an immediate consequence of the definition of commutativity.

\begin{lemma}\label{lclo}
For all $\al,\bt\in\mi(X)$, if $\al\bt=\bt\al$, then
$(\ima\al)\bt\subseteq\ima(\al)$ and $(\dom(\al))\bt^{-1}\subseteq\dom(\al)$.
\end{lemma}

\section{The Largest Commutative Inverse Semigroup in $\mi(X)$}\label{sinv}
\setcounter{equation}{0}
In this section, we will prove that the maximum order of a commutative inverse subsemigroup of $\mi(X)$
is $2^n$, and that the semilattice $E(\mi(X))$ of idempotents is the unique commutative inverse subsemigroup
of $\mi(X)$ of the maximum order (Theorem~\ref{tinv}).

An element $a$ of a semigroup $S$ is called \emph{regular} if $a=axa$
for some $x\in S$. If all elements of $S$ are regular, we say that
$S$ is a \emph{regular semigroup}. An element $a'\in S$ is called
an \emph{inverse} of $a\in S$ if $a=aa'a$ and $a'=a'aa'$. Since
regular elements are precisely those that have inverses (if $a=axa$
then $a'=xax$ is an inverse of $a$), we may define a regular semigroup
as a semigroup in which each element has an inverse \cite[p.~51]{Ho95}.
The most extensively studied subclass of the regular semigroups has been the class of inverse semigroups
(see  \cite{Pe84} and \cite[Chapter~5]{Ho95}). A semigroup
$S$ is called an \emph{inverse semigroup}
if every element of $S$ has exactly one inverse \cite[Definition~II.1.1]{Pe84}.
An alternative definition
is that $S$ is an inverse semigroup if it is a regular semigroup and its idempotents
(elements $e\in S$ such that $ee=e$) commute \cite[Theorem~5.1.1]{Ho95}.

A \emph{semilattice} is a commutative semigroup consisting entirely of idempotents. A semilattice
can also be defined as a partially ordered set $(S,\leq)$ such that the greatest lower
bound $a\wedge b$ exists for all $a,b\in S$. Indeed, if $S$ is a semilattice, then $(S,\leq)$, where
$\leq$ is a relation on $S$ defined by $a\leq b$ if $a=ab$, is a poset with $a\wedge b=ab$
for all $a,b\in S$. Conversely, if $(S,\leq)$ is a poset such that $a\wedge b$ exists for all $a,b\in S$,
then $S$ with multiplication $ab=a\wedge b$ is a semilattice. (See \cite[Proposition~1.3.2]{Ho95}.)
For a semigroup $S$,
denote by $E(S)$ the set of idempotents of $S$. The set $E(\mi(X))$ is a semilattice, which,
viewed as a poset, is isomorphic to the poset $(\mathcal{P}(X),\subseteq)$
of the power set $\mathcal{P}(X)$ under inclusion. 

For semigroups $S$ and $T$, we will write $S\cong T$ to mean that $S$ is isomorphic to $T$.

\begin{lemma}\label{lres3}
Let $S$ be a commutative semigroup in $\mi(X)$. Suppose there is a partition $\{A,B\}$ of $X$
such that $\al|_A,\bt|_A\in\mi(A)$ and $\al|_B,\bt|_B\in\mi(B)$ for all $\al,\bt\in S$.
Let
$S_{\!\ta}=\{\al|_{A}:\al\in S\}$ and $S_{\!\tb}=\{\al|_{B}:\al\in S\}$. Then:
\begin{itemize}
  \item [\rm(1)] $S_{\!\ta}$ is a commutative semigroup in $\mi(A)$ and $S_{\!\tb}$ is a commutative semigroup in $\mi(B)$.
  \item [\rm(2)] If $S$ is an inverse semigroup, then $S_{\!\ta}$ and $S_{\!\tb}$ are inverse semigroups.
  \item [\rm(3)] If $S$ is a maximal commutative semigroup in $\mi(X)$, then $S\cong S_{\!\ta}\times S_{\!\tb}$.
\end{itemize}
\end{lemma}
\begin{proof}
To prove (1), first note that $S_{\!\ta}$ is a subset of $\mi(A)$. It is closed under multiplication
since for all $\al,\bt\in S$, we have $\al\bt\in S$, and so, by Lemma~\ref{lres2}, $(\al|_A)(\bt|_A)=(\al\bt)|_A\in S_{\ta}$. Finally, $S_{\!\ta}$
is commutative by Lemma~\ref{lres2} and the fact that $S$ is commutative.
The proof for $S_{\!\tb}$ is the same.

To prove (2), suppose that $S$ is an inverse semigroup. Let $\al|_A\in S_{\!\ta}$, where $\al\in S$.
Since $S$ is a regular semigroup, there exists $\bt\in S$ such that $\al=\al\bt\al$. Then $\bt|_A\in S_{\!\ta}$
and, by Lemma~\ref{lres2}, $\al|_A=(\al\bt\al)|_A=(\al|_A)(\bt|_A)(\al|_A)$. Thus $\al|_A$ is a regular element of $S_{\!\ta}$,
and so $S_{\!\ta}$ is a regular semigroup. Hence $S_{\!\ta}$ is an inverse semigroup since it is a subsemigroup
of $\mi(A)$ and the idempotents in $\mi(A)$ commute. The proof for $S_{\!\tb}$ is the same.

To prove (3), suppose that $S$ is a maximal commutative semigroup in $\mi(X)$. Define a function
$\phi:S\to S_{\!\ta}\times S_{\!\tb}$ by $\al\phi=(\al|_A,\al|_B)$. Then $\phi$ is a homomorphism
since for all $\al,\bt\in S$, 
\[
(\al\bt)\phi=((\al\bt)|_A,(\al\bt)|_B)=((\al|_A)(\bt|_A),(\al|_B)(\bt|_B))=(\al|_A,\al|_B)(\bt|_A,\bt|_B)=(\al\phi)(\bt\phi).
\]
Further, for all $\al,\bt\in S$, $(\al|_A,\al|_B)=(\bt|_A,\bt|_B)$ implies $\al=\bt$ (since $\{A,B\}$ is a partition of $X$).
Thus $\phi$ is one-to-one. Let $(\sig,\mu)\in S_{\!\ta}\times S_{\!\tb}$. Then $\sig=\al|_A$ and $\mu=\bt|_B$ for some
$\al,\bt\in S$. Define $\gam\in\mi(X)$ by $\gam|_A=\al|_A$ and $\gam|_B=\bt|_B$. Let $\del\in S$. Then
$\al\del=\del\al$ and $\bt\del=\del\bt$, and so, by Lemma~\ref{lres2}, $(\gam|_A)(\del|_A)=(\al|_A)(\del|_A)=(\del|_A)(\al|_A)=(\del|_A)(\gam|_A)$
and $(\gam|_B)(\del|_B)=(\bt|_B)(\del|_B)=(\del|_B)(\bt|_B)=(\del|_B)(\gam|_B$). Hence $\gam\del=\del\gam$, which implies that $\gam\in S$
since $S$ is a maximal commutative semigroup in $\mi(X)$. Thus $\gam\phi=(\gam|_A,\gam|_B)=(\al|_A,\bt|_B)=(\sig,\mu)$,
and so $\phi$ is onto.
\end{proof}

A subgroup $G$ of $\sym(X)$ is called \emph{semiregular} if the identity is the only element of $G$
that fixes any point of $X$ \cite{Wi64}. It is easy to see that $G$ is semiregular if and only
if for every $\sig\in G$, all cycles in $\sig$ have the same length. If $G$ is a semiregular
subgroup of $\sym(X)$ with $n=|X|$, then the order of $G$ divides $n$ \cite[Proposition~4.2]{Wi64},
and so $|G|\leq n$.

We can now prove our main theorem in this section.

\begin{theorem}\label{tinv}
Let $X$ be a finite set with $n\geq1$ elements. Then:
\begin{itemize}
  \item [\rm(1)] If $S$ is a commutative inverse subsemigroup of $\mi(X)$, then $|S|\leq 2^n$.
  \item [\rm(2)] The semilattice $E(\mi(X))$ is the unique commutative inverse subsemigroup of $\mi(X)$ of order~$2^n$.
\end{itemize}
\end{theorem}
\begin{proof}
We will prove (1) and (2) simultaneously by induction on $n$. The statements are certainly true for $n=1$.
Let $n\geq2$ and suppose that (1) and (2) are true for every symmetric inverse semigroup on a set
with cardinality less than $n$.

Let $S$ be a maximal commutative inverse semigroup in $\mi(X)$. Let $G=S\cap\sym(X)$ and $T=S-G$.
If $G$ is a semiregular subgroup of $\sym(X)$ and $T=\{0\}$, then $|S|=|G|+1\leq n+1<2^n$ (since $n\geq2$).

Suppose $G$ is not semiregular or $T\ne\{0\}$. In the former case, $G$ (and so $S$) contains a permutation $\sig$
such that not all cycles of $\sig$ are of the same length. Suppose $T\ne\{0\}$. Let $0\ne\al\in T$ and let
$\al'$ be the inverse of $\al$ in $S$. Then $\al=\al\al'\al$ and $\vep=\al\al'$ is an idempotent.
Note that $\vep\ne1$ (since $\al\notin\sym(X)$) and $\vep\ne0$ (since $\al=\vep\al$ and $\al\ne0$).

Thus, in either case, by Lemmas~\ref{lres1} and~\ref{lres3}, there is a partition
$\{A,B\}$ of $X$ such that $S\cong S_{\!\ta}\times S_{\!\tb}$, where
$S_{\!\ta}$ is a commutative inverse semigroup in $\mi(A)$ and
$S_{\!\tb}$ is a commutative inverse semigroup in $\mi(B)$. Let $k=|A|$ and $m=|B|$. Then $1\leq k,m<n$ with $k+m=n$, and so,
by the inductive hypothesis,
$|S|=|S_{\!\ta}|\cdot|S_{\!\tb}|\leq 2^k\cdot2^m=2^{k+m}=2^n$.

Suppose that $S\ne E(\mi(X))$. Then, since $S$ is a maximal commutative inverse semigroup in $\mi(X)$, 
$S$ is not included in $E(\mi(X))$, and so it is not a semilattice. It follows that 
$S_{\!\ta}\ne E(\mi(A))$ or $S_{\!\tb}\ne E(\mi(B))$ (since $S\cong S_{\!\ta}\times S_{\!\tb}$ and the direct
product of two semilattices is a semilattice). We may assume that $S_{\!\ta}\ne E(\mi(A))$. By the inductive
hypothesis again, $\mi(A)<2^k$, and so $|S|=|S_{\!\ta}|\cdot|S_{\!\tb}|<2^k\cdot2^m=2^{k+m}=2^n$.

We have proved that $|S|\leq 2^n$ and if $S\ne E(\mi(X))$ then $|S|<2^n$. Statements (1) and (2) follow.
\end{proof}

\section{The Largest Commutative Nilpotent Semigroups in $\mi(X)$}\label{snil}
\setcounter{equation}{0}
In this section, we consider nilpotent semigroups in $\mi(X)$, that is,
the semigroups whose every element is a nilpotent. We determine the maximum order of a commutative 
nilpotent semigroup in $\mi(X)$,
and describe the commutative nilpotent semigroups in $\mi(X)$ of the maximum order
(Theorem~\ref{tnil}).

\begin{defi}\label{dnul}
{\rm
Let $X$ be a set with $n\geq2$ elements and let $\{K,L\}$ be a partition
of $X$. Denote by $S_{\!\tk\!,\tl}$ the subset of $\mi(X)$ consisting of
all nilpotents of the form $[x_1\,y_1]\jo\cdots\jo[x_r\,y_r]$, where $x_i\in K$, $y_i\in L$, and $0\leq r\leq\min\{|K|,|L|\}$.
}
\end{defi}

For example, let $n=4$, $X=\{1,2,3,4\}$, $K=\{1,2\}$, and $L=\{3,4\}$. Then
\[
S_{\!\tk\!,\tl}=\{0,[1\,3],[1\,4],[2\,3],[2\,4],[1\,3]\jo[2\,4],[1\,4]\jo[2\,3]\}.
\]

\begin{lemma}\label{lmax}
Any set $S_{\!\tk\!,\tl}$ from \textnormal{Definition~\ref{dnul}}
is a null semigroup of order $\sum_{r=0}^m\binom{m}{r}\binom{n-m}{r}r!$,
where $m=\min\{|K|,|L|\}$.
\end{lemma}
\begin{proof}
Let $\al,\bt\in S_{\!\tk\!,\tl}$ and suppose $x\in\dom(\al)$. Then $x\al\notin\dom(\bt)$
(since $x\al\in L$), and so $x\notin\dom(\al\bt)$. It follows that $\al\bt=0$.

Let $m=\min\{|K|,|L|\}$. Suppose $m=|K|$, so $|L|=n-m$. Let
$\al=[x_1\,y_1]\jo\cdots\jo[x_r\,y_r]$ be a transformation in $S_{\!\tk\!,\tl}$ of rank $r$.
Then, clearly, $0\leq r\leq m$. The domain of $\al$ can be selected in $\binom{m}{r}$ ways,
the image in $\binom{n-m}{r}$ ways, and the domain can be mapped to the image in $r!$ ways. It follows
that $S_{\!\tk\!,\tl}$ contains $\binom{m}{r}\binom{n-m}{r}r!$ transformations of rank $r$,
and so $|S|=\sum_{r=0}^m\binom{m}{r}\binom{n-m}{r}r!$. The result is also true when $m=|L|$ since
$S_{\!\tk\!,\tl}$ has the same order as $S_{\!\tl\!,\tk}$.
\end{proof}

\begin{defi}\label{dbal}
{\rm
A null semigroup  $S_{\!\tk\!,\tl}$ from \textnormal{Definition~\ref{dnul}}
such that $|K|=\nc$ and $L=n-\nc$, or vice versa, will be called
a \emph{balanced null semigroup}. By Lemma~\ref{lmax},
any balanced null semigroup has order
\begin{equation}\label{e1snil}
\lam_n=\sum_{r=0}^{\nc}\binom{\nc}{r}\binom{n-\nc}{r}r!.
\end{equation}
If $S_{\!\tk\!,\tl}$ is a balanced null semigroup, then
the monoid $S_{\!\tk\!,\tl}\cup\{1\}$ will be called a \emph{balanced null monoid}.

Note that $\lam_n$ from (\ref{e1snil}) is also defined for $n=1$, and that $\lam_1=1$
is the order of the trivial nilpotent semigroup $S=\{0\}$.
}
\end{defi}

Our objective is to prove that the maximum order of a commutative nilpotent subsemigroup of $\mi(X)$
is  $\lam_n$, and that, if $n\notin\{1,3\}$,
the balance null semigroups $S_{\!\tk\!,\tl}$ are the only commutative nilpotent subsemigroups
of $\mi(X)$ of order $\lam_n$ (Theorem~\ref{tnil}).
We will need some combinatorial lemmas, which we present now.

\begin{lemma}\label{lcom1}
For every $n\geq4$, $\lam_n=\lam_{n-1}+\nc\lam_{n-2}$.
\end{lemma}
\begin{proof}
Let $m=\nc$.
Consider a balanced null semigroup $S_{\!\tk\!,\tl}$, where $|K|=n-m$ and $|L|=m$.
Then $\lam_n=|S_{\!\tk\!,\tl}|$. Fix $x\in K$. Then $S_{\!\tk\!,\tl}=S_1\cup S_2$, where
$S_1=\{\al\in S_{\!\tk\!,\tl}:x\notin\dom(\al)\}$ and $S_2=\{\al\in S_{\!\tk\!,\tl}:x\in\dom(\al)\}$.
Then $S_1=S_{\!\tk-\{x\}\!,\tl}\subseteq\mi(X-\{x\})$ with $|K-\{x\}|=\left\lfloor\frac{n-1}{2}\right\rfloor$
and $|L|=(n-1)-\left\lfloor\frac{n-1}{2}\right\rfloor$. Thus $|S_1|=\lam_{n-1}$.

Let $\al\in S_2$. Then $\al=[x\,y]\jo\bt$, where $y\in L$ and $\bt\in S_{\!\tk-\{x\}\!,\tl-\{y\}}\subseteq\mi(X-\{x,y\})$
with  $|K-\{x\}|=(n-2)-\left\lfloor\frac{n-2}{2}\right\rfloor$ and $|L|=\left\lfloor\frac{n-2}{2}\right\rfloor$.
For a fixed $y\in L$, the mapping $\al=[x\,y]\jo\bt\to\bt$ is a bijection from $\{\al\in S_2:x\al=y\}$ to
$S_{\!\tk-\{x\}\!,\tl-\{y\}}$. Thus, since there are $|L|=m$ choices for $y$, we have
$|S_2|=m|S_{\!\tk-\{x\}\!,\tl-\{y\}}|=m\lam_{n-2}$. Hence
\[
\lam_n=|S_{\!\tk\!,\tl}|=|S_1|+|S_2|=\lam_{n-1}+\nc\lam_{n-2}
\]
since $m=\nc$.
\end{proof}

\begin{lemma} \label{lcom2}
Let $a,b$ be integers such that $1\leq a,b\leq n$, $a<\nc$, and $b=n-a$. Then
\[
\sum_{r=0}^{a}\binom{a}{r}\binom{b}{r}r!<\sum_{r=0}^{a+1}\binom{a+1}{r}\binom{b-1}{r}r!.
\]
\end{lemma}
\begin{proof}
Since  $a<\nc$, and $b=n-a$, we have $a<b$ and hence $a+1\leq b$. Let $0\leq r\leq a$. Then

\begin{align}
-b\leq -a-1 &\Rightarrow-br\leq -ar-r\notag \\
&\Rightarrow ba+b-br\leq ba+b-ar-r\notag \\
&\Rightarrow  b(a+1-r)\leq (b-r)(a+1)\notag \\
&\Rightarrow  \frac{b}{b-r}\leq \frac{(a+1)}{(a+1-r)}\notag \\
&\Rightarrow  \frac{b}{(b-r)(a-r)!(b-r-1)!}\leq \frac{(a+1)}{(a+1-r)(a-r)!(b-r-1)!}\notag \\
&\Rightarrow \frac{b}{(b-r)!(a-r)!}\leq \frac{(a+1)}{(a+1-r)!(b-r-1)!}\notag \\
&\Rightarrow  \frac{a!(b-1)!b}{(b-r)!(a-r)!}\leq \frac{a!(b-1)!(a+1)}{(a+1-r)!(b-r-1)!}\notag \\
&\Rightarrow  \frac{a!b!}{r!r!(b-r)!(a-r)!}\leq \frac{(b-1)!(a+1)!}{r!r!(a+1-r)!(b-r-1)!}\notag \\
&\Rightarrow  \frac{a!}{r!(a-r)!}\frac{b!}{r!(b-r)!}\leq \frac{(a+1)!}{r!(a+1-r)!}\frac{(b-1)!}{r!(b-r-1)!}\notag \\
&\Rightarrow  \binom{a}{r}\binom{b}{r}\leq \binom{a+1}{r}\binom{b-1}{r}\notag \\
&\Rightarrow  \binom{a}{r}\binom{b}{r}r!\leq \binom{a+1}{r}\binom{b-1}{r}r!\notag
\end{align}

Hence $\sum_{r=0}^{a}  \binom{a}{r}\binom{b}{r}r!\leq \sum_{r=0}^{a}\binom{a+1}{r}\binom{b-1}{r}r!$, and so
$\sum_{r=0}^{a}  \binom{a}{r}\binom{b}{r}r!< \sum_{r=0}^{a+1}\binom{a+1}{r}\binom{b-1}{r}r!$.
\end{proof}

\begin{lemma}\label{lcom3}
Let $n>10$. Then:
\begin{itemize}
  \item [\rm(1)] $\lam_n+1>2(\lam_{n-1}+1)$.
  \item [\rm(2)] For every positive integer $k$ such that $k\geq10$ and $n-k\geq10$,
\[
\lam_n+1>(\lam_k+1)(\lam_{n-k}+1).
\]
\end{itemize}
\end{lemma}
\begin{proof}
To prove (1), fix $a\in X$ and consider a partition $\{A,B\}$
of $X-\{a\}$ such that $|A|=\left\lfloor\frac{n-1}2\right\rfloor$ and $|B|=(n-1)-|A|$.
Note that $\lam_{n}=|S_{\!\ta\cup\{a\}\!,\tb}|$ and  $\lam_{n-1}=|S_{\!\ta,\tb}|$.
We will consider two cases.
\vskip 1mm
\noindent{\bf Case 1.} $n$ is even.
\vskip 1mm
In this case $|B|= |A|+1$, hence for every $\al\in S_{\!\ta,\tb}$, we can select an element 
$y_\al\in B-\ima(\al)$. 
Then the mapping $\phi:S_{\!\ta,\tb}\to S_{\!\ta\cup\{a\}\!,\tb}$ defined by
$\al\phi=\al\jo[a\,y_\al]$ is one-to-one with $\ima(\phi)\subseteq S_{\!\ta\cup\{a\}\!,\tb}-S_{\!\ta,\tb}$.
Since $n>10$, we can select $y_1,y_2\in B$ such that $y_1,y_2\ne y_\al$
where $\al=0$. Then $[a\,y_1],[a\,y_2]\in S_{\!\ta\cup\{a\}\!,\tb}-(S_{\!\ta,\tb}\cup\ima(\phi))$,
which implies 
\[
\lam_n=|S_{\!\ta\cup\{a\}\!,\tb}|\geq|S_{\!\ta,\tb}|+|\ima(\phi)|+|\{[a\,y_1],[a\,y_2]\}|=\lam_{n-1}+\lam_{n-1}+2>2\lam_{n-1}+1.
\]

\noindent{\bf Case 2.} $n$ is odd.
\vskip 1mm
Let $m=|A|=|B|=\frac{n-1}{2}$. By direct calculations, $\lam_{11}=4051$ and $2\lam_{10}+1=3093$.
So (1) is true for $n=11$. Suppose $n\geq13$ and note that $m\geq6$.
Denote by $J_{m-2}$ the set of transformations of $S_{\!\ta,\tb}$ of rank
at most $m-2$ and note that
\[
|J_{m-2}|=\lam_{n-1}-m!-\binom{m}{1}^2(m-1)!=\lam_{n-1}-(m+1)m!
\]
(since $S_{\!\ta,\tb}$ has $m!$ transformations of rank $m$, and $\binom{m}{1}^2(m-1)!$ transformations of rank $m-1$).
For every $\al\in J_{m-2}$, select two distinct elements $y_\al,z_\al\in B-\ima(\al)$
(possible since $|B|=m$ and $\rank(\al)\leq m-2$).
Then the mappings $\phi,\psi:J_{m-2}\to S_{\!\ta\cup\{a\}\!,\tb}$ defined by
$\al\phi=\al\jo[a\,y_\al]$ and $\al\psi=\al\jo[a\,z_\al]$ are one-to-one with 
$\ima(\phi)\cup\ima(\psi)\subseteq S_{\!\ta\cup\{a\}\!,\tb}-S_{\!\ta,\tb}$ and $\ima(\phi)\cap\ima(\psi)=\emptyset$.
Therefore, 
\begin{align}
\lambda_n&=|S_{\!\ta\cup\{a\}\!,\tb}|\geq|S_{\!\ta,\tb}|+|\ima(\phi)|+|\ima(\psi)|\notag\\
&=\lam_{n-1}+2(\lambda_{n-1}-(m+1)m!)=2\lam_{n-1}-2(m+1)m!+\lam_{n-1}\notag\\
&>2\lam_{n-1}-2(m+1)m!+\left(\binom{m}{0}^2m!+\binom{m}{1}^2 (m-1)!+\binom{m}{2}^2 (m-2)!\right)\notag\\
&=2\lam_{n-1}-2(m+1)m!+\left(m!+m \cdot m!+\frac{m(m-1)}{4}m!\right)\notag\\
&=2\lam_{n-1}+m!\left(-2m-2+1+m+\frac{m(m-1)}{4}\right)\notag\\
&=2\lambda_{n-1}+\frac{m!}{4}\left(m^2-5m-4\right)>2\lam_{n-1}+1,\notag
\end{align}
where the first strong inequality follows from the fact that $\lam_{n-1}=|S_{\!\ta,\tb}|$, $m\geq6$,
and the expression $\binom{m}{0}^2m!+\binom{m}{1}^2 (m-1)!+\binom{m}{2}^2(m-2)!$ only counts
the transformations in $S_{\!\ta,\tb}$ of ranks $m$, $m-1$, and $m-2$;
and the last strong inequality inequality follows from the fact that for $m\geq6$,
$\frac{m!}{4}\geq180$ and $m^2-5m-4\ge 2$. 

To prove (2), suppose $k\geq10$ and $n-k\geq10$.
We may assume that $k\leq n-k$. 
Consider a partition $\{A,B,C,D\}$ of $X$ such that 
\[
|A|=\left\lfloor\frac{n-k}2\right\rfloor,\,\, |B|=(n-k)-|A|,\,\, |D|=\left\lfloor\frac{k}2\right\rfloor,\,\,|C|=k-|D|.
\]
Then, either $|A|+|C|=\nc$ or $|B|+|D|=\nc$, and so
$\lam_n=|S_{\!\ta\cup\tc,\tb\cup\td}|$, $\lam_{n-k}=|S_{\!\ta,\tb}|$ and $\lam_k=|S_{\tc,\td}|$.

Let $S$ be the subsemigroup of $S_{\!\ta\cup\tc,\tb\cup\td}$ consisting of all $\al$ 
such that $\al|_A\in S_{\!\ta,\tb}$ and $\al|_C\in S_{\tc,\td}$.
We can construct a  bijection between $S$ and  $S_{\!\ta,\tb}\times S_{\tc,\td}$ as in the proof of Lemma \ref{lres3}, hence
$|S|=\lam_{n-k}\lam_k$. Since the inequality in~(2) is equivalent to $\lam_n>\lam_k\lam_{n-k}+\lam_k+\lam_{n-k}$,
it suffices to construct more then $2\lambda_{n-k}\ge\lambda_{n-k}+\lambda_{k}$ elements of
$S_{\!\ta\cup\tc,\tb\cup\td}-S$. We will consider two cases.
\vskip 1mm
\noindent{\bf Case 1.} $n-k$ is odd.
\vskip 1mm
In this case $|B|=|A|+1$, so for each $\al\in S_{\!\ta,\tb}$, we can select an element $b_\al\in B-\ima(\al)$.  
Now, for any pair $(c,\al)\in C\times S_{\!\ta,\tb}$, let $\al_c=\al\jo[c\,b_\al]$. 
It is clear that $\al_c\in S_{\!\ta\cup\tc,\tb\cup\td}-S$ and that the mapping $(\al,c)\to\al_c$ is one-to-one.
Since $k\geq10$, we have $|C|=k-\left\lfloor\frac{k}2\right\rfloor\ge 5$.
Thus, we have constructed $|C|\cdot|S_{\!\ta,\tb}|\geq5\lambda_{n-k}>2\lambda_{n-k}$ elements in $S_{\!\ta\cup\tc,\tb\cup\td}-S$.
\vskip 1mm
\noindent{\bf Case 2.} $n-k$ is even.
\vskip 1mm
Let $m=\frac{n-k}{2}$.  Note that for any $\al\in S_{\!\ta,\tb}$ of rank smaller then $m$, we can find $b_\al\in B-\ima(\al)$ 
and define $\al_c$ as in Case~1.  This construction yields $|C|(\lam_{n-k}-m!)\ge5(\lam_{n-k}-m!)$
distinct elements of $S_{\!\ta\cup\tc,\tb\cup\td}-S$. Since $m\geq5$, we have
\[
\lam_{n-k}=|S_{\!\ta,\tb}|>\binom{m}{0}\binom{m}{0}m!+\binom{m}{1}\binom{m}{1}(m-1)!> 2m!,
\]
where the first inequality follows from the fact that $\binom{m}{0}\binom{m}{0}m!+\binom{m}{1}\binom{m}{1}(m-1)!$
only counts the elements of $S_{\!\ta,\tb}$ of rank $m$ and $m-1$. Thus $3\lam_{n-k}>6m!$, and so
\[
|C|(\lam_{n-k}-m!)\ge5(\lam_{n-k}-m!)=3\lambda_{n-k}+2\lambda_{n-k}-5m!>6m!+2\lambda_{n-k}-5m!>2\lambda_{n-k}.
\]
The result follows.
\end{proof}

\begin{lemma}\label{lcom4}
If $n\geq6$, then $\lam_n>2\lam_{n-1}$.
\end{lemma}
\begin{proof}
If $n>10$, then $\lam_n>2\lam_{n-1}+1>2\lam_{n-1}$ by Lemma~\ref{lcom3}.
If $6\leq n\leq10$, then the result can be checked by direct calculations:
\[
\begin{tabular}{|c|c|c|c|c|c|c|c|c|c|}\hline
$n$&2&3&4&5&6&7&8&9&10\\\hline
$\lam_n$&2&3&7&13&34&73&209&501&1546\\\hline
$2\lam_{n-1}$&2&4&6&14&26&68&146&418&1002\\\hline
\end{tabular}
\]
\end{proof}

We begin the proof of Theorem~\ref{tnil} with introducing the following notation.

\begin{nota}\label{nabc}
{\rm 
Let $S$ be any commutative nilpotent subsemigroup of $\mi(X)$.
We define the following subset $C=C(S)$ of $X$:
\begin{equation}\label{e1nabc}
C=\{c\in X:c\in\dom(\al)\cap\ima(\bt)\mbox{ for some $\al,\bt\in S$}\}.
\end{equation}
For a fixed $c\in C$, we define
\begin{align}
A_c&=\{a\in X:a\al=c\mbox{ for some $\al\in S$}\},\notag\\
B_c&=\{b\in X:c\al=b\mbox{ for some $\al\in S$}\}.\notag
\end{align}
Note that $A_c$ and $B_c$ are not empty (by the definition of $C$) and that $A_c\cap B_c=\emptyset$.
(Indeed, if $a\in A_c\cap B_c$, then $a\al=c$ and $c\bt=a$ for some $\al,\bt\in S$, that is,
$\al=[\ldots a\,c\ldots]\jo\cdots$ and $\bt=[\ldots c\,a\ldots]\jo\cdots$. It then follows from
Proposition~\ref{pcen} that $\al\bt\ne\bt\al$,
which is a contradiction.)
}
\end{nota}

In the following lemmas, $S$ is a commutative nilpotent subsemigroup of $\mi(X)$ and
$C$ is the subset of $S$ defined by (\ref{e1nabc}). Our immediate objective is to obtain 
certain bounds on $|A_c|$ and $|B_c|$ (see Lemma~\ref{lbou}).

\begin{lemma}\label{lqabc}
Let $c\in C$, $a\in A_c$, and $b\in B_c$. Then: 
\begin{itemize}
  \item [\rm(1)] There is a unique $q=q(c,a,b)\in C$
such that for all $\al\in S$, if $a\al=c$, then $q\al=b$.
  \item [\rm(2)] For all $\bt\in S$, if $c\bt=b$, then $a\bt=q$, where $q=q(c,a,b)$ is the unique element from {\rm (1)}.
\end{itemize}
\end{lemma}
\begin{proof}
To prove (1), suppose $\al\in S$ with $a\al=c$, that is, $\al=[\ldots a\,c\ldots]\jo\cdots$. Since $b\in B_c$, $c\bt=b$ for some $\bt\in S$.
Since $c\in\dom(\bt)$, Proposition~\ref{pcen} implies that $a\in\dom(\bt)$. Let $q=a\bt$. Then
$q\al=(a\bt)\al=(a\al)\bt=c\bt=b$. Let $\al'\in S$ be such that $a\al'=c$. By the foregoing argument,
there exists $q'\in X$ such that $a\bt=q'$ and $q'\al'=b$. But then $q=a\bt=q'$, so $q$ is unique. 
Moreover, $q\in C$ since $q\in\dom(\al)\cap\ima(\bt)$.

To prove (2), suppose $\bt\in S$ with $c\bt=b$. Since $a\in A_c$, $a\al=c$ for some $\al\in S$.
But then, by the proof of (1), $a\bt=q$. 
\end{proof}

\begin{lemma}\label{ldif1}
Let $c\in C$, $a,a_1,a_2\in A_c$, and $b,b_1,b_2\in B_c$. Then:
\begin{itemize}
  \item [\rm(a)] If $q(c,a,b_1)=q(c,a,b_2)$, then $b_1=b_2$.
  \item [\rm(b)] If $q(c,a_1,b)=q(c,a_2,b)$, then $a_1=a_2$.
\end{itemize}
\end{lemma}
\begin{proof}
To prove (1), let $q=q(c,a,b_1)=q(c,a,b_2)$.
Since $a\in A_c$, there is $\al\in S$ such that $a\al=c$. But then, by Lemma~\ref{lqabc},
$b_1=q\al=b_2$. 
The proof of (2) is similar.
\end{proof}

We can now prove the lemma concerning the sizes of $A_c$ and $B_c$.

\begin{lemma}\label{lbou}
Suppose $C\ne\emptyset$. Then, there exists $c \in C$ such that one of the following conditions
holds:
\begin{itemize}
\item [\rm(a)] $|A_c|\geq2$ and $|B_c|\geq2$;
\item [\rm(b)] $|A_c|=1$ and $|B_c|\leq\nc$; or
\item [\rm(c)] $|B_c|=1$ and $|A_c|\leq\nc$.
\end{itemize}
\end{lemma}  
\begin{proof}
Suppose to the contrary that for every $c\in C$, none of (a)--(c) holds. Let $c\in C$.
Then, since (a) does not hold for $c$, $|A_c|=1$ or $|B_c|=1$. 

Suppose $|A_c|=1$, say $A_c=\{a\}$. 
Then, since (b) does not hold for $c$, $|B_c|>\nc$. Let $b\in B_c$.
We claim that $b\notin C$. 
Suppose to the contrary that $b\in C$. 
Construct elements $b_0,b_1,b_2,\ldots$ in $C\cap B_c$ as follows. Set $b_0=b$. Suppose $b_i\in C\cap B_c$ has been
constructed ($i\geq0$). Let $b_{i+1}$ be any element of $B_c$ such that $b_{i+1}\gam_i=b_i$ for some $\gam_i\in S$.
Then $b_{i+1}\in C$ as $b_{i+1} \in \dom(\gam_i) \cap B_c=\dom(\gam_i)\cap \{c\}S$.
If such an element $b_{i+1}$ does not exist, stop the construction. Note that the construction must stop
after finitely many steps. (Indeed, otherwise, since $X$ is finite, we would have $b_k=b_j$ with $k>j\geq0$.
But then $b_j\gam=b_k\gam=b_j$ for $\gam=\gam_{k-1}\gam_{k-2}\cdots\gam_j\in S$, which is impossible since $S$ consists of nilpotents.)
Thus, there exists $i\geq0$ such that $b_i\in C\cap B_c$ and no element of $B_c$ is mapped to $b_i$
by some transformation in $S$. 

Let $b'=b_i$ and note that $A_{b'}\subseteq X-B_c$.
Since $A_c=\{a\}$, $a\al=c$ for some $\al\in S$. Since $b'\in B_c$,
$c\bt=b'$ for some $\bt\in S$. Let $q=q(c,a,b')$. Then, by Lemma~\ref{lqabc}, $q\al=b'$,
and so $\{c,q\}\subseteq A_{b'}$. If $c\ne q$, then $|A_{b'}|\geq2$. Suppose $c=q$. Then
$a(\al\al)=c\al=q\al=b'$, and so $\{c,a\}\subseteq A_{b'}$. But $a\ne c$ (since $a\al=c$ and $\al$ is a nilpotent),
and we again have $|A_{b'}|\geq2$. On the other hand, since $A_{b'}\subseteq X-B_c$ and $|B_c|>\nc$,
we have $|A_{b'}|\leq\nc$. But $b'\in C$ with $2\leq|A_{b'}|\leq\nc$ contradicts our assumption (see the first sentence of the proof).

The claim has been proved. Hence, no element of $B_c$ is in $C$, that is, $C\subseteq X-B_c$. 
Now, by Lemma \ref{lqabc}, for each $b_i \in B_c$, there exists $q_i=q(c,a,b_i)\in C$ such that $a\in A_{q_i}$
and $b_i\in B_{q_i}$. Moreover, by Lemma~\ref{ldif1}, $q_i\ne q_j$ if $i\ne j$. But this is a contradiction
since $|B_c|>\nc>|X-B_c|\geq|C|$.

If $|B_c|=1$, we obtain a contradiction in a similar way. This concludes the proof.
\end{proof}

We continue the proof of Theorem~\ref{tnil} by considering two cases. First, we suppose that $S$ is a commutative
semigroup of nilpotents such that $C=\emptyset$, that is, there is no $c\in X$
such that $c\in\dom(\al)\cap\ima(\bt)$ for some $\al,\bt\in S$. Note that this implies
that each nonzero element of $S$ is a nilpotent of index $2$. 

\begin{prop}\label{pcase1}
Let $X$ be a set with $n\geq2$ elements and let $m=\nc$. Let $S$ be a commutative nilpotent subsemigroup
of $\mi(X)$ with $C=\emptyset$.
Suppose
$S\ne S_{\!\tk\!,\tl}$ for every balanced null semigroup $S_{\!\tk\!,\tl}$ (see {\rm Definition~\ref{dbal}}).
Then $|S|<\sum_{r=0}^m\binom{m}{r}\binom{n-m}{r}r!$.
\end{prop}
\begin{proof}
Let $A=\{x\in X:x\in\dom(\al)\mbox{ for some $\al\in S$}\}$ and $B=X-A$.
Since $C=\emptyset$, we have $A\cap\{y\in X:y\in\ima(\bt)\mbox{ for some $\bt\in S$}\}=\emptyset$,
and so $S\subseteq S_{\!\ta,\tb}$.

Suppose $|A|=m$. Then $S\ne S_{\!\ta,\tb}$ by the assumption, and so, by Lemma~\ref{lmax},
$|S|<|S_{\!\ta,\tb}|=\sum_{r=0}^m\binom{m}{r}\binom{n-m}{r}r!$.
Suppose $|A|<m$. Let $a=|A|$ and $b=|B|=n-a$. By Lemma~\ref{lmax} again,
\[
|S|\leq|S_{\!\ta,\tb}|=\sum_{r=0}^{a}\binom{a}{r}\binom{b}{r}r!.
\]
Applying Lemma~\ref{lcom2} $m-a$ times, we obtain
\[
|S|\leq\sum_{r=0}^{a}\binom{a}{r}\binom{b}{r}r!<\sum_{r=0}^m\binom{m}{r}\binom{n-m}{r}r!.
\]
Suppose $|A|>m$. Consider the semigroup $S'=\{\al^{-1}:\al\in S\}$ and note that $S'$ is a nilpotent
commutative semigroup with $C=C(S')=\emptyset$ and the corresponding set $A'$ included in the original set $B$.
Since $|A'|\leq m$, $|S|=|S'|<\sum_{r=0}^m\binom{m}{r}\binom{n-m}{r}r!$ by the foregoing argument.
\end{proof}

Second, we suppose that $S$ is a commutative nilpotent
subsemigroup of $\mi(X)$ such that $C\ne\emptyset$. Note that this is possible only if $n\geq3$.
Fix $c\in C$ that satisfies one of the conditions (1)--(3) from Lemma~\ref{lbou}.
Our objective is to prove that for all $n\geq3$,
\begin{equation}\label{eqsl1}
|S|\leq\lam_n=\sum_{r=0}^{\nc}\binom{\nc}{r}\binom{n-\nc}{r}r!.
\end{equation}

We will proceed by strong induction on $n=|X|$. Let $n=3$. Then the maximal commutative nilpotent semigroups in $\mi(X)$
are the balanced null semigroups $\{0,[i\,j],[i\,k]\}$ and  $\{0,[i\,k],[j\,k]\}$, and the cyclic
semigroups $\{0,[i\,j\,k],[i\,k]\}$, where $i,j,k$ are fixed, pairwise distinct, elements of $X$.
Thus (\ref{eqsl1}) is true for $n=3$.
\vskip 1mm
\noindent{\bf Inductive Hypothesis.} Let $n\geq4$ and suppose that (\ref{eqsl1}) is true whenever $3\leq|X|<n$.
\vskip 1mm
Consider the following subset of $S$:
\begin{equation}\label{eqsl2}
S_c=\{\al\in S:c\in\spa(\al)\}.
\end{equation}
Then $S-S_c$ is a commutative nilpotent subsemigroup of $\mi(X-\{c\})$.
If there is no $d\in X-\{c\}$ such that $d\in\dom(\al)\cap\ima(\bt)$ for some $\al,\bt\in S-S_c$,
then $|S-S_c|\leq \lam_{n-1}$ by Proposition~\ref{pcase1}. If such a $d\in X-\{c\}$ exists, then
$|S-S_c|\leq\lam_{n-1}$ by the inductive hypothesis. Thus, at any rate,
\begin{equation}\label{eqsl3}
|S-S_c|\leq\lam_{n-1}.
\end{equation}

We now want to find a suitable upper bound for the size of $S_c$ (Lemma~\ref{lbound}). To this end, we will map $S_c$ onto
a commutative subset $S_c^*$ of $\mi(X-\{c\})$ and analyze the preimages of the elements of $S_c^*$.

\begin{defi}\label{dscs}
{\rm 
For $\al \in S_c$ with $c\in\ima(\al)$, let $U_\al$ be the smallest subset of $X$ containing 
$c\al^{-1}$ and closed under all transformations $\gam^{-1}$ and 
$\al\del\al^{-1}$, where $\gam,\del\in S_c$.

For $\al \in S_c$ with $c\in\dom(\al)$, let $D_\al$ be the smallest subset of $X$ containing 
$c$ and closed under all transformations $\gam^{-1}$ and 
$\al\del\al^{-1}$, where $\gam,\del\in S_c$.

For $\al \in S_c$, define $\als\in\mi(X-\{c\})$ as follows:
\[
\als=
\left\{\begin{array}{ll}
\al|_{X-U_\al} & \mbox{if $c\in\ima(\al)-\dom(\al)$,}\\
\al|_{X-D_\al} & \mbox{if $c\in\dom(\al)-\ima(\al)$,}\\
\al|_{X-(D_\al\cup\, U_\al)} & \mbox{if $c\in\dom(\al)\cap\ima(\al)$.}
\end{array}\right.
\]
Let $S_c^*=\{\als:\al\in S_c\}$ and note that $S_c^*$ is a subset of $\mi(X-\{c\})$.
}
\end{defi}

We will need the following lemma about the sets $U_\al$ and $D_\al$.

\begin{lemma}\label{lpre}
Let $\al,\bt\in S_c$. Then:
\begin{itemize} 
  \item [\rm(1)] If $c\in\ima(\al)$, then $U_\al\subseteq\dom(\al)$. Moreover, if $c\in\ima(\bt)$ and
$c\al^{-1}=c\bt^{-1}$, then $U_\al=U_\bt$ and $x\al=x\bt$ for all $x\in U_\al$. 
  \item [\rm(2)] If $c\in\dom(\al)$, then $D_\al\subseteq\dom(\al)$. Moreover, if $c\in\dom(\bt)$ and
$c\al=c\bt$, then $D_\al =D_\bt$ and $x\al=x\bt$ for all $x\in D_\al$. 
  \item  [\rm(3)] If $c\in\dom(\al)\cap\ima(\al)$, then $D_\al=U_\al$. 
Moreover, if $c\in\ima(\bt)$ and $c\al^{-1}=c\bt^{-1}$, then $c\in\dom(\bt)$, $U_\al=U_\bt=D_\bt$,
and $x\al=x\bt$ for all $x\in U_\al$.
If $c\in\dom(\bt)$ and $c\al=c\bt$, then $c\in\ima(\bt)$, $U_\al=U_\bt=D_\bt$,
and $x\al=x\bt$ for all $x\in U_\al$.
\end{itemize}
\end{lemma}
\begin{proof}
To prove (1), suppose $c\in\ima(\al)$ and let  $a=c\al^{-1}$. Then clearly $a\in\dom(\al)$.
By Lemma~\ref{lclo}, $\dom(\al)$ is closed under $\gam^{-1}$ for all $\gam\in S_c$.
Let $x\in\dom(\al)$ and $\del\in S_c$ be such that $x(\al\del\al^{-1})$ is defined.
Since $x\al\in\ima(\al)$, we have $(x\al)\del\in\ima(\al)$ by Lemma~\ref{lclo},
and so $x(\al\del\al^{-1})=((x\al)\del)\al^{-1}\in\dom(\al)$. Thus $\dom(\al)$ is also closed
under $\al\del\al^{-1}$ for all $\del\in S_c$. It follows that $U_\al\subseteq\dom(\al)$.

Suppose $c\in\ima(\bt)$ and $c\al^{-1}=c\bt^{-1}$. Let $a=c\al^{-1}=c\bt^{-1}$. Let $x\in U_\bt$. We will prove that $x\in U_\al$ and 
$x\al=x\bt$ by induction on the minimum number of steps needed to generate $x$ from $a$.

If $x=a$, then $x\in U_\al$ and $x\al=x\bt$ since $x=a=c\bt^{-1}=c\al^{-1}$. Suppose 
$x=y\gam^{-1}$ for some $y\in U_\bt$ and $\gam\in S_c$. Then 
$y\in U_\al$ and $y\al=y\bt$ by the inductive hypothesis. Then $x=y\gam^{-1}\in U_\al$ by the definition of $U_\al$.
Further, $y\in C$ (since $y\in\dom(\al)\cap\ima(\gam)$), $x\in A_y$ (since $x\gam=y$), and $y\al\in B_y$. Since we also have $y\bt=y\al$,
Lemma~\ref{lqabc} implies
\[
x\al=q(y,x,y\al)=q(y,x,y\bt)=x\bt.
\]
Finally, suppose
$x=y(\bt\del\bt^{-1})$ for some $y\in U_\bt$ and $\del\in S_c$. Then
$y\in U_\al$ and $y\al=y\bt$ by the inductive hypothesis. Let $p=y(\al\del)$.
Then $y\al\in C$ (since $y\al\in\dom(\del)\cap\ima(\al))$), $y\in A_{y\al}$, and $p\in B_{y\al}$ (since $(y\al)\del=p$). Again, since $y\bt=y\al$,
Lemma~\ref{lqabc} implies
\[
p\al^{-1}=q(y\al,y,p)=q(y\bt,y,p)=p\bt^{-1}.
\]
Then $x=y(\bt\del\bt^{-1})=(y(\al\del))\bt^{-1}=p\bt^{-1}=p\al^{-1}=y(\al\del\al^{-1})$. It follows
that $x\in U_\al$ and $x\al=p=x\bt$. We have proved that $U_\bt\subseteq U_\al$ and $x\al=x\bt$ for all $x\in U_\bt$.
By symmetry, $U_\al\subseteq U_\bt$ and $x\al=x\bt$ for all $x\in U_\al$. We have proved (1). The proof of (2) is similar.

To prove (3), suppose $c\in\dom(\al)\cap\ima(\al)$, say $c\al=b$ and $a\al=c$. 
Then $a=c\al^{-1}\in D_\al$ and $c=a\al\al\al^{-1}\in U_\alpha$. Hence $U_\al=D_\al$ by the definitions
of $U_\al$ and $D_\al$. The remaining claims in (3) follow from (1) and (2).
\end{proof}

\begin{lemma}\label{lsccom}
Any two transformations in $S_c^*$ commute.
\end{lemma}

\begin{proof}
Let $\al,\bt \in S_c$. We want to prove that $\als\bts=\bts\als$. 
Let $x\in X-\{c\}$. Since $\al\bt=\bt\al$, both $\al\bt$ and $\bt\al$ are either defined at $x$ or undefined at $x$.
In the latter case, both $\als\bts$ and $\bts\als$ are undefined at $x$.

So suppose that $x(\al\bt)=x(\bt\al)$ exists.
If both $\als\bts$ and $\bts\als$ are defined at $x$, then $x(\als\bts)=x(\al\bt)=x(\bt\al)=x(\bts\als)$.
Hence, it suffices to show that 
\[
x(\als\bts)\mbox{ is undefined}\miff x(\bts\als)\mbox{ is undefined}.
\]
By symmetry, we may suppose that that $x(\als\bts)$ is undefined. We consider two possible cases.
\vskip 1mm
\noindent{\bf Case 1.} $x\als$ is undefined.
\vskip 1mm
Since we are working under the assumption that $x(\al\bt)$ exists (and so $x\al$ exists),
it follows from Definition~\ref{dscs} and Lemma~\ref{lpre} that $x\in K$, where $K=U_\al$ or $K=D_\al$. 
Since $x(\al\bt)$ exists, it is in $\ima(\al)$ by Lemma~\ref{lclo}, so $x(\al\bt\al^{-1})$ exists.
Hence $x(\al\bt\al^{-1})\in K$ by the definitions of $U_\al$ and $D_\al$. We have
$x(\al\bt\al^{-1}) =(x\bt\al)\al^{-1}=x\bt$, and so $x\bt\in K$. Thus $(x\bt)\als$ is undefined, and so 
$x(\bts\als)$ is undefined.
\vskip 1mm
\noindent{\bf Case 2.} $x\als$ is defined and $(x\als)\bts$ is undefined.
\vskip 1mm
This can only happen when $x\als=x\al$ is in $K$, where $K=U_\beta$ or $K=D_\bt$.
By the definitions of $U_\bt$ and $D_\bt$,
$x=(x\al)\al^{-1}\in K$ as well.  But then $x\bts$ is undefined, and hence
$x(\bts\als)$ is also undefined. 
\end{proof}

\begin{lemma}\label{lemp}
Let $\al\in S_c$. Then:
\begin{itemize}
  \item [\rm(1)] If $c\in\ima(\al)$, then $\spa(\als)\cap B_c=\emptyset$. 
  \item [\rm(2)] If $c\in\dom(\al)$, then $\spa(\als)\cap A_c=\emptyset$. 
\end{itemize}
\end{lemma}
\begin{proof}
To prove (1), let $c\in\ima(\al)$ and $b\in B_c$, that is, $c\gam=b$ for some $\gam\in S_c$.
Note that $b\in\ima(\al)$ by Lemma \ref{lclo}. Then, since $c\al^{-1}\in U_\al$, we have
$b\al^{-1}=(c\al^{-1})(\al\gam\al^{-1})\in U_\al$. Thus $b\al^{-1}\notin\dom(\als)$, and so $b\notin\ima(\als)$.
If $b\notin\dom(\al)$, then clearly $b\notin\dom(\als)$. Suppose $b\in\dom(\al)$. We have already
established that $b\al^{-1}\in U_\al$. Thus $b=(b\al^{-1})(\al\al\al^{-1})\in U_\al$, and so $b\notin\dom(\als)$.
We have proved (1). The proof of (2) is similar.
\end{proof}

We can now obtain an upper bound for the size of $S_c$.

\begin{lemma}\label{lbound}
Let $p=|A_c|$ and $t=|B_c|$. Then 
\[
|S_c|\leq(p-1)\lam_{n-t-1}+(t-1)\lam_{n-p-1}+2\lam_{n-p-t-1}.
\]
\end{lemma}
\begin{proof}
Let $A=A_c$, $B=B_c$, and consider the following subsets of $S_c^*$:
\begin{align}
F_A&=\{\als\in S_c^*:\spa(\als)\cap A\ne\emptyset\},\notag\\
F_B&=\{\als\in S_c^*:\spa(\als)\cap B\ne\emptyset\},\notag\\
F_0&=\{\als\in S_c^*:\spa(\als)\cap(A\cup B)=\emptyset\}.\notag
\end{align}
Suppose $\als\in F_A$. Then, by Lemma~\ref{lemp}, $c\in\ima(\al)-\dom(\al)$ and $\spa(\als)\cap B=\emptyset$.
Hence $\als\in\mi(X-(B\cup\{c\}))$.
Similarly, if $\als\in F_B$, then $c\in\dom(\al)-\ima(\al)$, $\spa(\als)\cap A=\emptyset$, and
$\als\in\mi(X-(A\cup\{c\}))$. If $\als\in F_0$, then clearly $\als\in\mi(X-(A\cup B\cup\{c\}))$.
Thus, $S_c^*=F_A\cup F_B\cup F_0$ and the sets $F_A$, $F_B$, and $F_0$ are pairwise disjoint.

By Lemma~\ref{lsccom}, $F_A$, $F_B$, and $F_0$ are sets of commuting transformations (as subsets of $S_c^*$).
Let $F$ be any subset of $S_c^*$ and denote by $\langle F\rangle$ the semigroup generated by $F$.
Then $\langle F\rangle$ is clearly commutative.
Suppose to the contrary that $\langle F\rangle$ is not a nilpotent semigroup. Then it contains a nonzero
idempotent, say  $\vep=\als_1\cdots\als_k$, where $\als_i\in F$. Let $x\in X$ be any element fixed by $\vep$.
Then $x(\als_1\cdots\als_k)=x$, and so $x(\al_1\cdots\al_k)=x$ since each $\als_i$ is a restriction of $\al_i$.
But this is a contradiction since $\al_1\cdots\al_k$ is a nilpotent as an element of $S$. Thus
$\langle F\rangle$ is a nilpotent semigroup.

Hence, by Proposition~\ref{pcase1} and the inductive
hypothesis applied to $\langle F_A\cup F_0\rangle\subseteq \mi(X-(B\cup\{c\}))$,
$\langle F_B\cup F_0\rangle\subseteq \mi(X-(A\cup\{c\}))$, and $\langle F_0\rangle\subseteq\mi(X-(A\cup B\cup\{c\}))$), we have
\begin{equation}\label{e1lbound}
|F_A|+|F_0|\leq\lam_{n-t-1},\,\,|F_B|+|F_0|\leq\lam_{n-p-1},\,\,|F_0|\leq\lam_{n-p-t-1}.
\end{equation}

Suppose $\als\in F_A$. Then $c\in\ima(\al)-\dom(\al)$, and so $a\al=c$ for some $a\in X$. Note that $a\in A$.
Fix $a_0\in\spa(\als)\cap A$.
Suppose to the contrary that $a_0=a$. Then $a_0\notin\dom(\als)$ since $a_0=a=c\al^{-1}\in U_\al$
and $\als=\al|_{X-U_\al}$. Hence $a_0\in\ima(\als)$, that is, $x\als=a_0=a$ for some $x\in\dom(\als)$.
But this is a contradiction since $x=a\al^{-1}\in U_\al$, and so $x\notin\dom(\als)$. We have proved that $a_0\ne a$.
Suppose $\als=\bts$. By the foregoing argument, there is $a'\in A$ such that $a'\bt=c$ and $a'\ne a_0$.
Moreover, if $a=a'$, then $\al=\bt$ by Lemma~\ref{lpre}. 

It follows that any $\als\in F_A$ has at most $p-1$
preimages under the mapping $^*$ (which correspond to the number of elements
from the set $A-\{a_0\}$ that $\al$ can map to $c$ if $\als\in F_A$). By similar arguments, any $\al\in F_B$ has at most $t-1$ preimages under~$^*$,
and any $\als\in F_0$ has at most $p+t$ preimages under~$^*$. These considerations about the number of preimages
that an element of $S_c^*$ can have, together with (\ref{e1lbound}), give
\begin{align}
|S_c|&\leq(p-1)|F_A|+(t-1)|F_B|+(p+t)|F_0|\notag\\
&=(p-1)(|F_A|+|F_0|)+(t-1)(|F_B|+|F_0|)+2|F_0|\notag\\
&\leq (p-1)\lam_{n-t-1}+(t-1)\lam_{n-p-1}+2\lam_{n-p-t-1},\notag
\end{align}
which completes the proof.
\end{proof}

The following proposition will finish our inductive proof of (\ref{eqsl1}). The proposition is stronger than what we need in this section,
but we will also use it in the proof of the general case.

\begin{prop}\label{pcase2}
Let $X$ be a set with $n\geq4$. Let $S$ be a commutative nilpotent subsemigroup
of $\mi(X)$ with $C\ne\emptyset$. Then:
\begin{itemize}
  \item [\rm(1)] If $n\leq7$, then $|S|<\lam_n$.
  \item [\rm(2)] If $n\geq8$, then $|S|<\lam_n-n$.
\end{itemize}
\end{prop}
\begin{proof}
We have checked that (1) is true by direct calculations using GRAPE \cite{So06},
which is a package for GAP \cite{Scel92}. For $n\in\{4,5,6,7\}$, we have calculated the orders of the
maximal commutative nilpotent semigroups and the number of semigroups of each order.
The following table contains the maximum order of a commutative nilpotent semigroup
(row 2) and the number of commutative nilpotent semigroups of the maximum order.

\[
\begin{tabular}{|c|c|c|c|c|}\hline
$n$&4&5&6&7\\\hline
Max order&7&13&34&73\\\hline
No of sgps of max order&6&20&20&70\\\hline
\end{tabular}
\]
The numbers in the second row of the table are $\lam_4$, $\lam_5$, $\lam_6$, and $\lam_7$ (see the table in Lemma~\ref{lcom4}).
The numbers in the third row are $\binom{4}{2}$, $2\binom{5}{2}$, $\binom{6}{3}$, and $2\binom{7}{3}$. This means that the commutative
nilpotent semigroups of the maximum order are the balanced null semigroups  $S_{\!\tk\!,\tl}$
since, for $m=\nc$, there are $\binom{n}{m}$ such semigroups if $n$ is even,
and $2\binom{n}{m}$ such semigroups if $n$ is odd (see the proof of Theorem~\ref{tgen}). Since $C=\emptyset$ for each semigroup  $S_{\!\tk\!,\tl}$ (balanced or not), (1) follows.

To prove (2), suppose $n\geq8$. Let $c\in C$ be an element that satisfies one of the conditions (1)--(3)
from Lemma~\ref{lbou}. Let $p=|A_c|$ and $t=|B_c|$. By (\ref{eqsl3}) and Lemma~\ref{lbound},
\begin{equation}\label{e1}
|S|=|S-S_c|+|S_c|\leq\lam_{n-1}+(p-1)\lam_{n-t-1}+(t-1)\lam_{n-p-1}+2\lam_{n-p-t-1}.
\end{equation}
We consider four possible cases.
\vskip 1mm
\noindent{\bf Case 1.} $p\geq2$ and $t\geq2$.
\vskip 1mm
By (\ref{e1}),
\begin{align}
|S|&\leq\lam_{n-1}+(p-1)\lam_{n-t-1}+(t-1)\lam_{n-p-1}+ 2\lam_{n-p-t-1}\notag\\
\label{e11}&\leq\lam_{n-1}+(p-1)\lam_{n-3}+(t-1)\lam_{n-3}+2\lam_{n-5}\\
\label{e21}&\leq\lam_{n-1}+(n-3)\lam_{n-3}+2\lam_{n-5},
\end{align}
where (\ref{e11}) follows from $n-t-1,n-p-1\leq n-3$ and $n-p-t-1\leq n-5$, and (\ref{e21}) from $p+t\leq n-1$
(so $p+t-2\leq n-3$).
For $n=8$ and $n=9$, $\lam_{n-1}+(n-3)\lam_{n-3}+2\lam_{n-5}<\lam_n-n$ by direct calculations:
\[
\begin{tabular}{|c|c|c|}\hline
$n$&8&9\\\hline
$\lam_n-n$&201&492\\\hline
$\lam_{n-1}+(n-3)\lam_{n-3}+2\lam_{n-5}$&144&427\\\hline
\end{tabular}
\]
For $n\geq10$, $\lam_{n-5}>n$ (see the table in Lemma~\ref{lcom4}), and so
\begin{align}
|S|&\leq\lam_{n-1}+(n-3)\lam_{n-3}+2\lam_{n-5}\notag\\
\label{e2a1}&<\lam_{n-1}+(n-3)\lam_{n-3}+3\lam_{n-5}-n\\
\label{e31}&<\lam_{n-1}+(n-3)\lam_{n-3}+\mbox{$\frac{3}{4}$}\lam_{n-3}-n\\
&<\lam_{n-1}+(n-2)\lam_{n-3}-n\notag\\
\label{e41}&\leq \lam_{n-1}+\mbox{$\nc$}\lam_{n-2}-n\\
\label{e51}&=\lam_n-n,
\end{align}
where  (\ref{e2a1}) follows from $\lam_{n-5}>n$ when $n\geq10$ (see Lemma~\ref{lcom4} and the table in its proof),
(\ref{e31}) from $\lam_{n-3}>4\lam_{n-5}$ when $n\geq10$ (see Lemma~\ref{lcom4} and the table in its proof),
(\ref{e41}) from $\lam_{n-2}>2\lam_{n-3}>\frac{n-2}{\nc}\lam_{n-3}$ when $n\geq8$ (see Lemma~\ref{lcom4}),
and (\ref{e51}) from Lemma~\ref{lcom1}.
\vskip 1mm
\noindent{\bf Case 2.} $p=1$ and $t=1$.
\vskip 1mm
Then, by (\ref{e1}),
\begin{align}
|S|&\leq\lam_{n-1}+2\lam_{n-3}\notag\\
\label{e12}&<\lam_{n-1}+3\lam_{n-3}-n\\
\label{e22}&<\lam_{n-1}+\mbox{$\frac{3}{2}$}\lam_{n-2}-n\\
&<\lam_{n-1}+\mbox{$\nc$}\lam_{n-2}-n\notag\\
\label{e32}&=\lam_n-n,
\end{align}
where (\ref{e12}) follows from $\lam_{n-3}>n$ when $n\geq8$, (\ref{e22}) from $\lam_{n-2}>2\lam_{n-3}$ when $n\geq8$ (see Lemma~\ref{lcom4}),
and (\ref{e32}) from Lemma~\ref{lcom1}.
\vskip 1mm
\noindent{\bf Case 3.} $p=1$ and $2 \le t\leq\nc$.
\vskip 1mm
Again by (\ref{e1}),
\begin{align}
|S|&\leq\lam_{n-1}+(t-1)\lam_{n-2}+2\lam_{n-t-2}\notag\\
&\leq\lam_{n-1}+(\mbox{$\nc-1$})\lam_{n-2}+2\lam_{n-4}\notag\\
\label{e13}&\leq\lam_{n-1}+(\mbox{$\nc-1$})\lam_{n-2}+3\lam_{n-4}-n+1\\
\label{e23}&<\lam_{n-1}+(\mbox{$\nc-1$})\lam_{n-2}+\mbox{$\frac34$}\lam_{n-2}-n+1\\
&=\lam_{n-1}+\mbox{$\nc$}\lam_{n-2}-\mbox{$\frac14$}\lam_{n-2}-n+1\notag\\
\label{e33}&<\lam_{n-1}+\mbox{$\nc$}\lam_{n-2}-n\\
\label{e43}&=\lam_n-n,
\end{align}
where (\ref{e13}) follows from $\lam_{n-4}\geq n-1$ when $n\geq8$ (see Lemma~\ref{lcom4} and the table in its proof), 
(\ref{e23}) from $\lam_{n-2}>4\lam_{n-4}$ when $n\geq8$ (see Lemma~\ref{lcom4} and the table in its proof),
(\ref{e33}) from $\frac14\lam_{n-2}>1$ for $n\geq6$,
and (\ref{e43}) from Lemma~\ref{lcom1}.
\vskip 1mm
\noindent{\bf Case 4.} $2\le p\leq\nc$ and $t=1$.
\vskip 1mm
This case is symmetric to Case~3.
\end{proof}

The inductive proof of (\ref{eqsl1}) is complete. As a bonus, we have Proposition~\ref{pcase2}. We can now
prove the main theorem of this section.

\begin{theorem}\label{tnil}
Let $X$ be a set with $n\geq1$ elements and let $m=\nc$. Then:
\begin{itemize}
   \item [\rm(1)] The maximum cardinality of a commutative nilpotent subsemigroup of $\mi(X)$ is
\[ 
\lam_n=\sum_{r=0}^m\binom{m}{r}\binom{n-m}{r}r!.
\]
  \item [\rm(2)] If $n\notin\{1,3\}$, then the only commutative nilpotent subsemigroups of $\mi(X)$ of order
$\lam_n$ are the balanced null semigroups $S_{\!\tk\!,\tl}$.
\end{itemize}
\end{theorem}
\begin{proof}
Let $S$ be a commutative nilpotent subsemigroup of $\mi(X)$. If $n=1$, then $S=\{0\}$, so $|S|=1=\lam_1$.
Let $n\geq2$. If $S=S_{\!\tk\!,\tl}$ is a balanced null semigroup, then
$|S|=\lam_n$ by Lemma~\ref{lmax}. Suppose $S$ is not one of the balanced null semigroups $S_{\!\tk\!,\tl}$.
If $C=\emptyset$, then $|S|<\lam_n$ by Proposition~\ref{pcase1}.
Suppose $C\ne\emptyset$. If $n=3$, then 
$S=\langle[i\,j\,k]\rangle=\{0,[i\,j\,k],[i\,k]\}$, where $i,j,k$ are pairwise distinct elements of $X$,
so $|S|=3=\lam_3$. If $n\geq4$, then $|S|<\lam_n$ by Proposition~\ref{pcase2}.
The result follows. 
\end{proof}

\section{The Largest Commutative Semigroups in $\mi(X)$}\label{sgen}
In this section, we determine the maximum order of a commutative subsemigroup of $\mi(X)$,
and describe the commutative subsemigroups of $\mi(X)$ of the maximum order (Theorem~\ref{tgen}).

\begin{lemma}\label{leix1}
Let $X$ be a set with $n<10$ elements. Suppose $S$ is a commutative subsemigroup of $\mi(X)$
such that $S\ne E(\mi(X))$, where $E(\mi(X))$ is the semilattice of idempotents of $\mi(X)$.
Then $|S|<2^n$.
\end{lemma}
\begin{proof}
The lemma is vacuously true when $n=1$. It is also true when $n=2$ since then the only maximal
commutative subsemigroups of $\mi(X)$ other than $E(\mi(X))$ are $\sym(X)\cup\{0\}$ and $\{0,1,[i\,j]\}$,
where $i$ and $j$ are distinct elements of $X$. Let $n\geq3$ and suppose, as the inductive hypothesis,
that the result is true whenever $|X|<n$. Let $G=S\cap\sym(X)$ and $T=S-G$.

Suppose $G$ is a semiregular subgroup of $\sym(X)$ and $T$ is a nilpotent semigroup. Then $|G|\leq n$ (since $G$ is semiregular)
and $|T|\leq\lam_n$ (by Theorem~\ref{tnil}). Thus $|S|\leq \lam_n+n<2^n$, where the latter inequality follows from
the table below.
\[
\begin{tabular}{|c|c|c|c|c|c|c|c|}\hline
$n$&3&4&5&6&7&8&9\\\hline
$\lam_n+n$&6&11&18&40&80&217&510\\\hline
$2^n$&8&16&32&64&128&256&512\\\hline
\end{tabular}
\]

Suppose $G$ is not a semiregular group or $T$ is not a nilpotent semigroup. 
Then, by Lemmas~\ref{lres1} and~\ref{lres3},
there is a partition $\{A,B\}$ of $X$ such that $S\cong S_{\!\ta}\times S_{\!\tb}$, where $S_{\!\ta}$ is a commutative subsemigroup
of $\mi(A)$ and $S_{\!\tb}$ is a commutative subsemigroup of $\mi(B)$. 
If $S\subseteq E(\mi(X))$, then $|S|<|E(\mi(X))|=2^n$. Suppose $S$ is not included in $E(\mi(X))$.
Then at least one of $S_{\!\ta}$ and $S_{\!\tb}$, say $S_{\!\ta}$,
must contain an element that is not an idempotent.
Let $k=|S_{\!\ta}|$. By the inductive hypothesis, $|S_{\!\ta}|<2^k$ and $|S_{\!\tb}|\leq2^{n-k}$, and so
$|S|=| S_{\!\ta}|\cdot| S_{\!\tb}|<2^k\cdot2^{n-k}=2^n$.
\end{proof}

\begin{lemma}\label{leix2}
Let $n=|X|\geq5$. Suppose $S=G\cup T$ is a commutative subsemigroup of $\mi(X)$ such that $G$ is a 
nontrivial semiregular subgroup of $\sym(X)$ and $T$ is a subsemigroup of 
$S_{\!\ta,\tb}$, where $\{A,B\}$ is a partition of $X$. Then $|S|<\lam_n+1$.
\end{lemma}
\begin{proof}
Let $k=|A|$, so $|B|=n-k$.
We have $|G|\leq n$ (since $G$ is semiregular) and $|T|\leq|S_{\!\ta,\tb}|\leq\lam_n$ (by Proposition~\ref{pcase1}). 
If $k=1$ or $k=n-1$, then $|T|\leq |S_{\!\ta,\tb}|=n$, and so
$|S|\leq n+n=2n<\lam_n+1$ since $n\geq5$ (see the table in Lemma~\ref{lcom4}). 

Suppose $1<k<n-1$. The semigroup $S_{\!\ta,\tb}$ contains $|A|\cdot|B|=k(n-k)$ nilpotents $[x\,y]$.
Let $\sig$ be a nontrivial element of $G$. Then no nilpotent $[x\,y]$ commutes with $\sig$ (by Proposition~\ref{pcen}), and so
such a nilpotent cannot be in $T$. Thus $|T|\leq|S_{\!\ta,\tb}|-k(n-k)\leq\lam_n-k(n-k)$. But, since $n\geq5$ and $1<k<n-1$,
we have $k(n-k)\geq n$ by elementary algebra, and so
\[
|S|=|G|+|T|\leq n+\lam_n-k(n-k)\leq n+\lam_n-n=\lam_n<\lam_n+1.
\]
\end{proof}

We can now prove the main theorem of the paper regarding largest commutative subsemigroups of $\mi(X)$.

\begin{theorem}\label{tgen}
Let $X$ be a set with $n\geq1$ elements and let $m=\nc$. Then:
\begin{itemize}
  \item [\rm(1)] If $n<10$, then the maximum cardinality of a commutative subsemigroup of $\mi(X)$ is $2^n$, and
the semilattice $E(\mi(X))$ is the unique commutative subsemigroup of $\mi(X)$ of order~$2^n$.
  \item [\rm(2)] Suppose $n\geq10$. Then the maximum cardinality of a commutative subsemigroup of $\mi(X)$ is
\[ 
\lam_n+1=\sum_{r=0}^m\binom{m}{r}\binom{n-m}{r}r!+1.
\]
\begin{itemize}
  \item [\rm(a)] If $n$ is even, then there are exactly $\binom{n}{m}$, pairwise isomorphic,
commutative subsemigroups of $\mi(X)$ of order $\lam_n+1$, namely
the balanced null monoids $S_{\!\tk\!,\tl}\cup\{1\}$.
  \item [\rm(b)] If $n$ is odd, then there are exactly $2\binom{n}{m}$, pairwise isomorphic,
commutative subsemigroups of $\mi(X)$ of order $\lam_n+1$, namely
the balanced null monoids $S_{\!\tk\!,\tl}\cup\{1\}$.
\end{itemize}
\end{itemize}
\end{theorem}
\begin{proof}
Statement (1) follows immediately from Lemma~\ref{leix1} and the fact that if $|X|=n$, then $|E(\mi(X))|=2^n$.

To prove (2), suppose $n\geq10$. Each of the balanced null monoids $S_{\!\tk\!,\tl}\cup\{1\}$
has order $\lam_n+1$ by Lemma~\ref{lmax}. If $n$ is even, then $|K|=|L|=m$, and so there are
$\binom{n}{m}$ balanced null semigroups $S_{\!\tk\!,\tl}$
(since we have $\binom{n}{m}$ choices for $K$ and $L=X-K$ is determined when $K$ has been selected). 
If $n$ is odd, then the number doubles since we have
$\binom{n}{m}$ such semigroups when $|K|=m$ and another $\binom{n}{m}$ when $|K|=n-m$.

Let $S$ be a commutative subsemigroup of $\mi(X)$ that is different from the balanced null monoids $S_{\!\tk\!,\tl}\cup\{1\}$.
Our objective is to prove that
\begin{equation}\label{eq1}
|S|<\lam_n+1=\sum_{r=0}^m\binom{m}{r}\binom{n-m}{r}r!+1.
\end{equation}
Proceeding by induction on $n=|X|$, we suppose that the statement is true for every $X$ with $10\leq|X|<n$.
Let $G=S\cap\sym(X)$ and $T=S-G$.

Suppose $G$ is a semiregular subgroup of $\sym(X)$ and $T$ is a nilpotent semigroup. If $G$ is trivial, then $T$ is not a 
balanced null semigroup, and hence $|S| < \lam_n +1$ by Theorem \ref{tnil}. So assume that $G \ne \{1\}$.
Let  
\[
C=\{c\in X:\mbox{$c\in\dom(\al)\cap\ima(\bt)$ for some $\al,\bt\in T$}\}. 
\]
If $C=\emptyset$,
then $T\subseteq S_{\!\ta,\tb}$, where $\{A,B\}$ is a partition of $X$, and so $|S|<\lam_n+1$
by Lemma~\ref{leix2}. Suppose $C\ne\emptyset$.
Then $|G|\leq n$ (since $G$ is semiregular)
and $|T|<\lam_n-n$ (by Proposition~\ref{pcase2}). Thus $|S|=|G|+|T|<n+\lam_n-n=\lam_n<\lam_n+1$.

Suppose $G$ is not a semiregular subgroup of $\sym(X)$ or $T$ is a not a nilpotent semigroup.
Then, by Lemmas~\ref{lres1} and~\ref{lres3},
there is a partition $\{A,B\}$ of $X$ such that $S\cong S_{\!\ta}\times S_{\!\tb}$, where $S_{\!\ta}$ is a commutative subsemigroup
of $\mi(A)$ and $S_{\!\tb}$ is a commutative subsemigroup of $\mi(B)$.
Notice that $1\leq|A|,|B|<|X|=n$. We may assume that $|A|\leq|B|$. 
Let $k=|A|$. Then $1\leq k<n$ and $|B|=n-k$. We consider three possible cases.
\vskip 1mm
\noindent{\bf Case 1.} $k<10$ and $n-k<10$.
\vskip 1mm
Then, by Lemma~\ref{leix1}, $|S_{\!\ta}|\leq2^k$ and $|S_{\!\tb}|\leq2^{n-k}$, and so
\[
|S|=|S_{\!\ta}|\cdot|S_{\!\tb}|\leq2^k\cdot2^{n-k}=2^n<\lam_n+1,
\]
where the last inequality is true since $n\geq10$ (see Lemma~\ref{lcom4} and the table in its proof).
\vskip 1mm
\noindent{\bf Case 2.} $k<10$ and $n-k\geq10$.
\vskip 1mm
Then, $|S_{\!\ta}|\leq2^k$ (by Lemma~\ref{leix1}) and $|S_{\!\tb}|\leq\lam_{n-k}+1$ (by Theorem~\ref{tnil} and the inductive hypothesis).
Thus, by (1) of Lemma~\ref{lcom3},
\[
|S|=|S_{\!\ta}|\cdot|S_{\!\tb}|\leq2^k(\lam_{n-k}+1)<\lam_n+1.
\]
\noindent{\bf Case 3.}  $k\geq10$ and $n-k\geq10$.
\vskip 1mm
Then, by Theorem~\ref{tnil} and the inductive hypothesis, $|S_{\!\ta}|\leq\lam_k+1$ and $|S_{\!\tb}|\leq\lam_{n-k}+1$.
Thus, by (2) of Lemma~\ref{lcom3},
\[
|S|=|S_{\!\ta}|\cdot|S_{\!\tb}|\leq(\lam_k+1)(\lam_{n-k}+1)<\lam_n+1.
\]
Hence, in all cases, $|S|<\sum_{r=0}^m\binom{m}{r}\binom{n-m}{r}r!+1$, which concludes the proof of (2).
\end{proof}

It follows from Theorem~\ref{tgen} that every symmetric inverse semigroup $\mi(X)$
has, up to isomorphism, a unique commutative subsemigroup of maximum order.
In comparison, the symmetric group $\sym(X)$ has, up to isomorphism, a unique
abelian subgroup of maximum order if $|X|=3k$ or $|X|=3k+2$, and two abelian subgroups of maximum order
if $|X|=3k+1$ \cite[Theorem~1]{BuGo89}.

\section{The Clique Number and Diameter of $\cg(\mi(X))$}\label{scdg}

In this section, we determine the clique number of the commuting graph of $\mi(X)$
and the diameter of the commuting graph of every nonzero ideal of $\mi(X)$. The exception is 
the case of $\cg(\mi(X))$ when $n=|X|$ is odd and composite, and not a prime power, where we are only able to say
that the diameter is either $4$ or $5$. 

Let $\Gamma$ be a simple graph, that is, $\Gamma=(V,E)$, where $V$ is a finite
non-empty set of vertices and $E\subseteq\{\{u,v\}:u,v\in V, u\ne v\}$ is a set of edges.
We will write $u-v$ to mean that $\{u,v\}\in E$.
(If $\cg(S)$ is the commuting graph of a semigroup $S$, then for all vertices
$a$ and $b$ of $\cg(S)$, $a-b$ if and only if $a\ne b$ and $ab=ba$.)

A subset $K$ of $V$ is called a \emph{clique} in $\Gamma$ if $u-v$ for all distinct $u,v\in K$. The \emph{clique number}
of $\Gamma$ is the largest integer $r$ such that $\Gamma$ has a clique $K$ with $|K|=r$.

Let $u,w\in V$. A \emph{path} in $\Gamma$ of length $m-1$ ($m\geq1$) from $u$ to $w$
is a sequence of pairwise distinct vertices
$u=v_1,v_2,\ldots,v_m=w$ such that $v_i-v_{i+1}$ for every $i\in\{1,\ldots,m-1\}$.
The \emph{distance} between vertices $u$ and $w$, denoted $d(u,w)$, is the smallest integer $k\geq0$ such that
there is a path of length $k$ from $u$ to $w$. 
If there is no path from $u$ to $w$, we
say that the distance between $u$ and $w$ is infinity, and write $d(u,w)=\infty$. The maximum
distance $\max\{d(u,w):u,w\in V\}$ between vertices of $\Gamma$ is called the \emph{diameter}
of $\Gamma$. Note that the diameter of $\Gamma$ is finite if and only if $\Gamma$ is connected.

It follows easily from Proposition~\ref{pcen} that the only central elements of $\mi(X)$
are the zero and identity transformations. Therefore, the following result is an immediate
corollary of Theorem~\ref{tgen}. (Note that if $|X|=1$, then $\mi(X)$ is a commutative semigroup.)

\begin{cor}\label{ccli}
Let $X$ be a set with $n\geq2$ elements and let $m=\nc$. Then:
\begin{itemize}
  \item [\rm(1)] If $n<10$, then the clique number of the commuting graph of $\mi(X)$ is $2^n-2$.
  \item [\rm(2)] If $n\geq10$, then the clique number of the commuting graph of $\mi(X)$ is
\[ 
\lam_n-1=\sum_{r=0}^m\binom{m}{r}\binom{n-m}{r}r!-1.
\]
\end{itemize}
\end{cor}

It is well known (see \cite[Exercises~5.11.2 and~5.11.4]{Ho95}) 
that $\mi(X)$ has exactly $n+1$ ideals, $J_0,J_1,\ldots,J_n$, where
\[
J_r=\{\al\in\mi(X):\rank(\al)\leq r\}
\]
for $0\leq r\leq n$.
Each ideal $J_r$ is principal and any $\al\in\mi(X)$ of rank $r$ generates $J_r$.
The ideal $J_0=\{0\}$ consists of the zero transformation.
Our next objective is to find the diameter of the commuting graph of every proper nonzero ideal $\mi(X)$.

\begin{lemma}\label{ldia1}
Let $n\geq2$. Suppose $\al\in\mi(X)-\{0,1\}$ is not an $n$-cycle or a nilpotent of index $n$.
Then there exists an idempotent $\vep\in\mi(X)-\{0,1\}$ such that $\rank(\vep)\leq\rank(\al)$ and $\al\vep=\vep\al$.
\end{lemma}
\begin{proof}
Let $\al=\rho_1\jo\cdots\jo\rho_k\jo\tau_1\jo\cdots\jo\tau_m$ be the decomposition of $\al$
as in Proposition~\ref{pdec}. Suppose $k,m\geq1$ (that is, $\al$ contains at least one cycle and at least one chain).
Then there is an integer $p\geq1$ such that $\vep=\al^p$ is an idempotent different from $0$ and $1$.
Clearly, $\al\vep=\vep\al$.

Suppose $k=0$. Since $\al$ is not a nilpotent of index $n$, $\spa(\tau_1)\ne X$. Let $\vep$
be the idempotent with $\dom(\vep)=\spa(\tau_1)$. Then $\vep\ne1$ (since $\spa(\tau_1)\ne X$), $\vep\ne0$
(since $\spa(\tau_1)\ne\emptyset$), and $\al\vep=\vep\al$ by Proposition~\ref{pcen}.
Suppose $m=0$. Then $k\geq2$ since $\al$ is not an $n$-cycle. Then $\al\vep=\vep\al$
for the idempotent $\vep$ with $\dom(\vep)=\spa(\rho_1)$.

Note that in all cases, $\rank(\vep)\leq\rank(\al)$.
\end{proof}

\begin{lemma}\label{ldia2}
Let $n\geq4$. Suppose $\al,\bt\in\mi(X)-\{0,1\}$ such that neither $\al$ nor $\bt$ is an $n$-cycle.
Then in the commuting graph $\cg(\mi(X))$, there is a path from $\al$ to $\bt$ of length at most $4$
such that all vertices in the path have rank at most $\max\{\rank(\al),\rank(\bt)\}$.
\end{lemma}
\begin{proof}
Suppose neither $\al$ nor $\bt$ is a nilpotent of index $n$. Then, by Lemma~\ref{ldia1}, there are idempotents
$\vep_1,\vep_2\in\mi(X)-\{0,1\}$ such that $\rank(\vep_1)\leq\rank(\al)$, $\rank(\vep_2)\leq\rank(\bt)$,
$\al-\vep_1$, and $\vep_2-\bt$. Since idempotents in $\mi(X)$ commute,
$\al-\vep_1-\vep_2-\bt$.

Suppose $\al=[y_1\,y_2\ldots\,y_n]$ is a nilpotent of index $n$ and $\bt$ is not a nilpotent of index $n$.
Let $\vep_1$ be the idempotent with $\dom(\vep_1)=\{y_1,y_n\}$ (note that $\rank(\vep_1)\leq\rank(\al)$)
and $\vep_2$ be an idempotent different from
$0$ and $1$ such that $\rank(\vep_2)\leq\rank(\bt)$ and $\vep_2-\bt$ (such an idempotent exists by Lemma~\ref{ldia1}).
Then $\al-[y_1\,y_n]-\vep_1-\vep_2-\bt$.

Finally, suppose $\al=[y_1\,y_2\ldots\,y_n]$ and $\bt=[x_1\,x_2\ldots\,x_n]$ are nilpotents of index $n$.
If $\{y_1,y_n\}\cap\{x_1,x_n\}=\emptyset$, then $[y_1\,y_n]$ and $[x_1\,x_n]$ commute, and so
$\al-[y_1\,y_n]-[x_1\,x_n]-\bt$. Suppose $\{y_1,y_n\}\cap\{x_1,x_n\}\ne\emptyset$. Then, since $n\geq4$,
there is $z\in X-\{y_1,y_n,x_1,x_n\}$. Let $\vep$ be the idempotent with $\dom(\vep)=\{z\}$.
Then, by Proposition~\ref{pcen}, $\al-[y_1\,y_n]-\vep-[x_1\,x_n]-\bt$.
\end{proof}

\begin{lemma}\label{ldia4}
Let $n\geq2$. Suppose $\al,\bt\in\mi(X)-\{0,1\}$ with $\al\bt=\bt\al$. Then
\begin{itemize}
  \item [\rm(1)] If $\al=[x_1\ldots\,x_n]$ is a nilpotent of index $n$, 
then there is $q\in\{1,\ldots,n-1\}$ such that $\bt=\al^q$.
  \item [\rm(2)] If $\al=(x_0\,x_1\ldots\,x_{n-1})$ is an $n$-cycle,
then there is $q\in\{1,\ldots,n-1\}$ such that $\bt=\al^q$.
\end{itemize}
\end{lemma}
\begin{proof}
Suppose $\al=[x_1\ldots\,x_n]$.
Since $\bt\notin\{0,1\}$, it follows by Proposition~\ref{pcen} that there is $t\in\{1,\ldots,n-1\}$ 
such that $\dom(\bt)\cap\{x_1,\ldots,x_n\}=\{x_1,\ldots,x_t\}$ and
\[
x_1\bt=x_{n-t+1},\,\,x_2\bt=x_{n-t+2},\ldots,\,\,x_t\bt=x_n.
\]
Thus $\bt=\al^q$, where $q=n-t$, and $q\notin\{0,n\}$ (since $1\leq t\leq n-1$).
We have proved (1).

Suppose $\al=(x_0\,x_1\ldots\,x_{n-1})$. 
Since $\bt\ne0$, $\{x_0,x_1,\ldots,x_{n-1}\}\subseteq\dom(\bt)$ by Proposition~\ref{pcen}.
Let $x_q=x_0\bt$, where $q\in\{0,1,\ldots,n-1\}$, and note that $q\ne0$ since $\al\ne1$. 
Then, by Proposition~\ref{pcen}, $x_i\bt=x_{q+i}$ for every $i\in\{0,\ldots,n-1\}$
(where $x_{q+i}=x_{q+i-n}$ if $q+i\geq n$). Thus $\bt=\al^q$.
We have proved (2).
\end{proof}

\begin{lemma}\label{ldia5}
Let $n\geq3$. Then there are nilpotents $\al,\bt\in\mi(X)$ of index $n$ such that $d(\al,\bt)=4$.
\end{lemma}
\begin{proof}
Let $\al=[x_1\,x_2\ldots\,x_k\,y_1\,y_2\ldots\,y_m]$ and $\bt=[y_m\ldots\,y_2\,y_1\,x_k\ldots\,x_2\,x_1]$, 
where $k+m=n$ and $k=\lceil\frac{n}{2}\rceil$. If $n\geq4$, then $d(\al,\bt)\leq4$ by Lemma~\ref{ldia2}.
If $n=3$, then $\al=[x_1\,x_2\,y_1]-[x_1\,y_1]-\vep-[y_1\,x_1]-[y_1\,x_2\,x_1]=\bt$,
where $\vep$ is the idempotent with $\dom(\vep)=\{x_2\}$, so $d(\al,\bt)\leq4$.

Note that $\al$ and $\bt$ do not commute, so $d(\al,\bt)\geq2$. Suppose $\al-\gam-\del-\bt$ is a path
from $\al$ to $\bt$ of length $3$. By Lemma~\ref{ldia4}, $\gam=\al^p$ and $\del=\bt^q$ for some
$p,q\in\{1,\ldots,n-1\}$. We may assume that $p\geq k$. (If not, then there exists an integer $t$ such that
$k\leq pt\leq n-1$, and so $\al^p$ can be replaced with $\al^{pt}=(\al^p)^t$ in the path.)
Similarly, we may assume that $q\geq k$. Then
\[
\al^p=[x_1\,y_i]\jo[x_2\,y_{i+1}]\jo\cdots\jo[x_{m-i+1}\,y_m]\mbox{ and }
\bt^q=[y_m\,x_j]\jo[y_{m-1}\,x_{j-1}]\jo\cdots\jo[y_{m-j+1}\,x_1],
\]
for some $i\in\{1,\ldots,m\}$ and $j\in\{1,\ldots,k\}$ (with $j\in\{1,\ldots,k-1\}$ when $n$ is odd).
But then $\al^p$ and $\bt^q$ do not commute (since $x_{m-i+1}(\al^p\bt^q)=x_j$ and $x_{m-i+1}\notin\dom(\bt^q\al^p)$),
which is a contradiction.

We have proved that there is no path from $\al$ to $\bt$ of length $3$. But then there is no path from $\al$ to $\bt$
of length $2$ either since any such path would have the form $\al-\al^p-\bt$ (and then $\al-\al^p-\bt^2-\bt$ would
be a path of length $3$) or $\al-\bt^q-\bt$ (and then $\al-\al^2-\bt^q-\bt$ would be a path of length $3$).
It follows that $d(\al,\bt)=4$.
\end{proof}

\begin{lemma}\label{ldia6}
Let $n\geq3$ and $\nca<r<n-1$. Then there are $\al,\bt\in J_r$ such that for every nonzero $\gam\in\mi(X)$,
if $\al-\gam-\bt$, then $\gam=1$.
\end{lemma}
\begin{proof}
Consider a nilpotent $\al=[x\,z_1\ldots\,z_{r-1}\,y]$ of rank $r$ (possible since $r<n-1$). 
Since $r>\nca$, we have $r>\frac{n-1}2$, and so $2r\geq n$.
Therefore, there are pairwise distinct elements $w_1,\ldots,w_{r-1}$ of $X$ such that
$\{x,y,x_1,\ldots,x_{r-1},w_1,\ldots,w_{r-1}\}=X$. Let $\bt=[y\,w_1\ldots\,w_{r-1}\,x]\in J_r$,
and suppose $0\ne\gam\in\mi(X)$ is such that $\al-\gam-\bt$. We want to prove that $\gam=1$.

Since $\gam\ne0$ and $\spa(\al)\cup\spa(\bt)=X$, we have $\dom(\gam)\cap\spa(\al)\ne\emptyset$
or $\dom(\gam)\cap\spa(\bt)\ne\emptyset$. We may assume that $\dom(\gam)\cap\spa(\al)\ne\emptyset$.
Then, since $\al\gam=\gam\al$, $x\in\dom(\gam)$ by Proposition~\ref{pcen}. Since $\bt\gam=\gam\bt$,
$x\in\dom(\gam)$ and Proposition~\ref{pcen} imply that $\spa(\bt)\subseteq\dom(\gam)$ and $\gam$ maps
$\bt$ onto a terminal segment of some chain in $\bt$. But $\bt$ is a single chain, so $\gam$
must map $\bt$ onto $\bt$, which is only possible if $\gam$ fixes every element of $\spa(\bt)$.
We now know that $\dom(\gam)\cap\spa(\bt)\ne\emptyset$. By the foregoing argument, with the roles of $\al$ and $\bt$ reversed,
we conclude that $\gam$ must also fix every element of $\spa(\al)$. Hence $\gam=1$.
\end{proof}

We can now determine the diameter of $\cg(J_r)$ for every $r<n$.

\begin{theorem}\label{tpro}
Let $n=|X|\geq3$ and let $J_r$ be a proper nonzero ideal of $\mi(X)$. Then:
\begin{itemize}
  \item [\rm(1)] The diameter of $\cg(J_{n-1})$ is $4$.
  \item [\rm(2)] If $\nca<r<n-1$, then the diameter of $\cg(J_r)$ is $3$.
  \item [\rm(3)] If $1\leq r\leq\nca$, then the diameter of $\cg(J_r)$ is $2$.
\end{itemize}
\end{theorem}
\begin{proof}
We first note that for every $r\in\{1,\ldots,n-1\}$, the only central element of $J_r$ is $0$. 

To prove (1), observe that $J_{n-1}=\mi(X)-\sym(X)$. The diameter of $\cg(J_{n-1})$
is at least
$4$ by Lemma~\ref{ldia5}. If $n\geq4$, then it is at most $4$ by Lemma~\ref{ldia2}. 

Let $n=3$ and let $\al,\bt\in J_{n-1}-\{0\}$.
If $\al$ or $\bt$ is not a nilpotent of index $3$, then $d(\al,\bt)\leq4$ by the proof of Lemma~\ref{ldia2}
(where the assumption $n\geq4$ was only used in the case when both $\al$ and $\bt$ were nilpotents of index $n$).

Let $\al=[x\,y\,z]$ and $\bt$ be distinct nilpotents of index $3$. We want to show that $d(\al,\bt)\leq4$. Since $\al-[x\,z]$,
it suffices to show that $d([x\,z],\bt)\leq3$. If $\bt=[x\,z\,y]$, then $[x\,z]-[x\,y]-[x\,z\,y]$;
if $\bt=[y\,x\,z]$, then $[x\,z]-[y\,z]-[y\,x\,z]$; if $\bt=[y\,z\,x]$, then $[x\,z]-[y\,z]-[y\,x]-[y\,z\,x]$;
if $\bt=[z\,x\,y]$, then $[x\,z]-[x\,y]-[z\,y]-[z\,x\,y]$; finally, if $\bt=[z\,y\,x]$, then
$[x\,z]-\vep-[z\,x]-[z\,y\,x]$, where $\vep$ is the idempotent with $\dom(\vep)=\{y\}$. Thus $d(\al,\bt)\leq4$,
which concludes the proof of (1).

To prove (2), suppose $\nca<r<n-1$. Then the diameter of $\cg(J_r)$ is at least $3$ by Lemma~\ref{ldia6}.
Let $\al,\bt\in J_r$. Since $r<n-1$, neither $\al$ nor $\bt$ is an $n$-cycle or a nilpotent of index $n$.
Thus, by Lemma~\ref{ldia1}, there are idempotents $\vep_1,\vep_2\in J_r-\{0\}$
such that $\al\vep_1=\vep_1\al$ and $\bt\vep_2=\vep_2\bt$. Since the idempotents in $\mi(X)$ commute,
we have $\al-\vep_1-\vep_2-\bt$, so the diameter of $\cg(J_r)$ is at most $3$.

To prove (3), suppose $1\leq r\leq\nca$. Then the diameter of $\cg(J_r)$ is at least $2$ since 
for any distinct $x,y\in X$, the nilpotents $[x\,y]$ and $[y\,x]$ (which are in $J_r$
since $r\geq1$) do not commute. Let $\al,\bt\in J_r-\{0\}$. We have $r\leq\nca\leq\frac{n-1}2$, and so
$2r\leq n-1<n$. Therefore,
\[
|\ima(\al)\cup\ima(\bt)|\leq|\ima(\al)|+|\ima(\bt)|\leq r+r=2r<n,
\]
and so there is $x\in X$ such that $x\notin\ima(\al)\cup\ima(\bt)$. By the same argument,
there is $y\in X$ such that $y\notin\dom(\al)\cup\dom(\bt)$.
If $x=y$, then $\al-\vep-\bt$, where $\vep$ is the idempotent with $\dom(\vep)=\{x\}$.
If $x\ne y$, then $\al-[x\,y]-\bt$. Thus, the diameter of $\cg(J_r)$ is at most $2$.
\end{proof}

We now want to prove that if $n\geq4$ is even, then the diameter of $\cg(\mi(X))$ is $4$.

\begin{defi}\label{dali}
{\rm 
Let $\gam,\del\in\mi(X)$. We say that $\gam$ and $\del$ are \emph{aligned} if there exists an integer $r\geq2$
and pairwise distinct elements $a_1,\ldots,a_r,c_1,\ldots,c_{r-1},b_1$ of $X$ such that
\begin{align}
\gam&=(a_1\,b_1)\jo(a_2\,c_1)\jo(a_3\,c_2)\jo\cdots\jo(a_{r-1}\,c_{r-2})\jo(a_r\,c_{r-1}),\notag\\
\del&=(a_1\,c_1)\jo(a_2\,c_2)\jo(a_3\,c_3)\jo\cdots\jo(a_{r-1}\,c_{r-1})\jo(a_r\,b_1).\notag
\end{align}
}
\end{defi}

The following lemma follows immediately from Definition~\ref{dali}

\begin{lemma}\label{lali}
Let $\gam,\del\in\mi(X)$ be aligned. Then, with the notation from {\rm Definition~\ref{dali}},
\[
\gam-(a_1\,a_2\ldots\,a_r)\jo(b_1\,c_1\ldots\,c_{r-1})-\del.
\]
\end{lemma}

\begin{lemma}\label{lncy}
Let $n=2k=|X|\geq4$ be even. Suppose $\al,\bt\in\sym(X)$ are joins of $k$ cycles of length~$2$ with no cycle in common.
Then $\al=\gam\jo\al'$ and $\bt=\del\jo\bt'$, where $\gam$ and $\del$ are aligned.
\end{lemma}
\begin{proof}
Select any cycle $(a_1\,b_1)$ in $\al$. Then $\bt$ has a cycle $(a_1\,c_1)$ with $c_1\ne b_1$
(since $\al$ and $\bt$ have no cycle in common). Continuing, $\al$ must have a cycle $(a_2\,c_1)$,
and so $\bt$ must have either a cycle $(a_2\,b_1)$ or a cycle $(a_2\,c_2)$ with $c_2\ne b_1$. 
In the latter case, $\al$ must have a cycle $(a_3\,c_2)$, and so $\bt$ must have a cycle $(a_3\,b_1)$ or
a cycle $(a_3\,c_3)$ with $c_3\ne b_1$. This process must terminate after at most $k$, say $r$, steps. That is,
at step $r$, we will obtain a cycle $(a_{r}\,c_{r-1})$ in $\al$ and a cycle $(a_r\,b_1)$ in $\bt$. Hence
\begin{align}
\al&=(a_1\,b_1)\jo(a_2\,c_1)\jo(a_3\,c_2)\jo\cdots\jo(a_{r-1}\,c_{r-2})\jo(a_r\,c_{r-1})\jo\al',\notag\\
\bt&=(a_1\,c_1)\jo(a_2\,c_2)\jo(a_3\,c_3)\jo\cdots\jo(a_{r-1}\,c_{r-1})\jo(a_r\,b_1)\jo\bt',\notag
\end{align}
where $\al'=\bt'=0$ if $r=k$. The proof is completed by the observation that
$\gam=(a_1\,b_1)\jo(a_2\,c_1)\jo(a_3\,c_2)\jo\cdots\jo(a_{r-1}\,c_{r-2})\jo(a_r\,c_{r-1})$
and $\del=(a_1\,c_1)\jo(a_2\,c_2)\jo(a_3\,c_3)\jo\cdots\jo(a_{r-1}\,c_{r-1})\jo(a_r\,b_1)$
are aligned.
\end{proof}

\begin{lemma}\label{ldia3}
Let $n\geq6$ be composite. Suppose $\al,\bt\in\mi(X)-\{0,1\}$ such that $\al$ is an $n$-cycle and $\bt$ is not an $n$-cycle.
Then $d(\al,\bt)\leq4$.
\end{lemma}
\begin{proof}

Suppose $\bt$ is not a nilpotent of index $n$. Since $n$ is composite, there is a divisor
$k$ of $n$ with $1<k<n$. Then $\al^k\in\mi(X)-\{0,1\}$ is not an $n$-cycle. Thus, by Lemma~\ref{ldia1}, there are
idempotents $\vep_1,\vep_2\in\mi(X)-\{0,1\}$ such that $\al^k-\vep_1$ and $\vep_2-\bt$.
Then $\al-\al^k-\vep_1-\vep_2-\bt$, and so $d(\al,\bt)\leq4$.

Suppose $\bt=[x_1\,x_2\ldots\,x_n]$ is a nilpotent of index $n$. Let $k$ be the largest proper
divisor of $n$. Then $\al=\rho_1\jo\cdots\jo\rho_k$, where each $\rho_i$
is a cycle of length $\frac{n}{k}$. Since $n\geq6$, we have $k>2$. Thus, there exists $t\in\{1,\ldots,k\}$
such that $x_1,x_n\notin\spa(\rho_t)$. Let $\vep$ be the idempotent with $\dom(\vep)=\spa(\rho_t)$.
Then $\vep\ne0,1$ and, by Proposition~\ref{pcen}, $\al-\al^k-\vep-[x_1\,x_n]-\bt$. Hence $d(\al,\bt)\leq4$.
\end{proof}

\begin{theorem}\label{tdie}
Let $n=|X|\geq4$ be even. Then the diameter of $\cg(\mi(X))$ is $4$.
\end{theorem}
\begin{proof}
Let $\al,\bt\in\mi(X)-\{0,1\}$. We will prove that $d(\al,\bt)\leq4$.
If neither $\al$ nor $\bt$ is an $n$-cycle, then $d(\al,\bt)\leq4$ by Lemma~\ref{ldia2}. 

Suppose $\al$ is an $n$-cycle and $\bt$ is not an $n$-cycle. If $n\geq6$, then $d(\al,\bt)\leq4$ by 
Lemma~\ref{ldia3}. If $n=4$ and $\bt$ is not a nilpotent of index $4$, then $d(\al,\bt)\leq4$
again by Lemma~\ref{ldia3} (where the assumption $n\geq6$ was only used when $\bt$ was a nilpotent of index $n$).
Let $n=4$, $\al=(x\,y\,z\,w)$, and $\bt=[a\,b\,c\,d]$. Then $\al^2=(x\,z)\jo(y\,w)$, $\bt^3=[a\,d]$, and so it suffices to find
a path of length $2$ from $(x\,z)\jo(y\,w)$ to $[a\,d]$. If $\{a,d\}=\{x,z\}$ or $\{a,d\}=\{y,w\}$, then
$(x\,z)\jo(y\,w)-\vep-[a\,d]$, where $\vep$ is the idempotent with $\dom(\vep)=\{a,d\}$.
Otherwise, we may assume that $a=x$ and $d=w$, and then $(x\,z)\jo(y\,w)-[x\,y]\jo[z\,w]-[x\,w]=[a\,d]$.
Hence $d(\al,\bt)\leq4$.

Suppose $\al$ and $\bt$ are $n$-cycles. Then for $k=n/2$, $\al^k$ and $\bt^k$ are joins
of $k$ cycles of length $2$. Therefore, it suffices to find a path of length $2$ from $\al^k$ to $\bt^k$.
If $\al^k$ and $\bt^k$ have a cycle in common, say $(a\,b)$, then $\al^k-\vep-\bt^k$, where
$\vep$ is the idempotent with $\dom(\vep)=\{a,b\}$. 

Suppose $\al^k$ and $\bt^k$ have no common cycle.
Then $\al^k=\gam\jo\al'$ and $\bt^k=\del\jo\bt'$, where $\gam,\al',\del,\bt'$
are as in Lemma~\ref{lncy}. By Lemma~\ref{lali},
there is $\eta\in\mi(X)$ such that $\spa(\eta)=\spa(\gam)=\spa(\del)$ and $\gam-\eta-\del$.
It follows that
\[
\al^k=\gam\jo\al'-\eta-\del\jo\bt'=\bt^k.
\]

We have proved that $d(\al,\bt)\leq4$ for all $\al,\bt\in\mi(X)$, which shows that the diameter
of $\cg(\mi(X))$ is at most $4$. Since the diameter of $\cg(\mi(X))$ is at least
$4$ by Lemma~\ref{ldia5}, the proof is complete.
\end{proof}

Suppose $n=2$, say $X=\{x,y\}$. Then the commuting graph $\cg(\mi(X))$ has one edge, $(x)-(y)$
(recall that in our notation $(x)$ is the idempotent with domain $\{x\}$),
and three isolated vertices, $(x\,y)$, $[x\,y]$, and $[y\,x]$. Hence, the diameter of $\cg(\mi(X))$ is $\infty$.

The following proposition and Theorem~\ref{tpow} partially solve the problem of finding the diameter of $\cg(\mi(X))$
when $n$ is odd.

\begin{prop}\label{pdio}
Let $n=|X|\geq3$ be odd. Then:
\begin{itemize}
   \item [\rm(1)] If $n$ is prime, then $\cg(\mi(X))$ is $\infty$.
   \item [\rm(2)] If $n$ is composite, then the diameter of $\cg(\mi(X))$ is either $4$ or $5$.
\end{itemize}
\end{prop}
\begin{proof}Suppose $n=p$ is an odd prime. Consider a $p$-cycle $\al=(x_0\,x_1\ldots x_{p-1})$ and let $\bt\in\mi(X)-\{0,1\}$ with
$\al\bt=\bt\al$. By Lemma~\ref{ldia4}, $\bt=\al^q$ for some $q\in\{1,\ldots,p-1\}$. Thus, since $p$ is prime, 
$\bt$ is also a $p$-cycle. It follows that
if $\gam$ is a vertex of $\cg(\mi(X))$ that is not a $p$-cycle, then there is no path in $\cg(\mi(X))$ from
$\al$ to $\gam$. Hence $\cg(\mi(X))$ is not connected, and so the diameter of $\cg(\mi(X))$ is~$\infty$. We have proved (1).

Suppose $n$ is odd and composite (so $n\geq9$). Let $\al,\bt\in\mi(X)-\{0,1\}$. If $\al$ or $\bt$ is not an
$n$-cycle, then $d(\al,\bt)\leq4$ by Lemmas~\ref{ldia2} and~\ref{ldia3}. Suppose $\al$ and $\bt$ are $n$-cycles.
Let $k$ be a proper divisor of $n$ ($1<k<n$). Then $\al=\rho_1\jo\cdots\jo\rho_k$
and $\bt=\sig_1\jo\cdots\jo\sig_k$, where each $\rho_i$ and each $\sig_i$
is a cycle of length $\frac{n}{k}$. Let $\vep_1$ and $\vep_2$ be the idempotents with $\dom(\vep_1)=\spa(\rho_1)$
and $\dom(\vep_2)=\spa(\sig_1)$. Then, $\al^k,\bt^k\ne1$ (since $k<n$), $\vep_1,\vep_2\ne1$ (since $k>1$),
and $\al-\al^k-\vep_1-\vep_2-\bt^k-\bt$. Hence $d(\al,\bt)\leq5$, and so the diameter of $\cg(\mi(X))$ is at most $5$.
On the other hand, the diameter of $\cg(\mi(X))$ is at least $4$ by Lemma~\ref{ldia5}. We have proved (2).
\end{proof}

We will now prove that when $n=p^k$ is a power of an odd prime $p$, with $k\geq2$, then the diameter of $\cg(\mi(X))$ is $5$.

\begin{defi}\label{dhag}
{\rm 
Let $\al=\rho_1\jo\rho_2\jo\cdots\jo\rho_k\in\sym(X)$ and let $\gam\in\mi(X)$ with $\al\gam=\gam\al$.
We define a partial transformation $\hag$ on the set $A=\{\rho_1,\rho_2,\ldots,\rho_k\}$ of cycles of $\al$ by:
\begin{align}
\dom(\hag)&=\{\rho_i\in A:\spa(\rho_i)\cap\dom(\gam)\ne\emptyset\},\notag\\
\rho_i\hag&=\mbox{the unique $\rho_j\in A$ such that $(\spa(\rho_i))\gam=\spa(\rho_j)$.}\notag
\end{align}
Note that $\hag$ is well defined and injective by Proposition~\ref{pcen}. 
}
\end{defi}

The case of $n=3^2=9$ is special and we consider it in the following lemma.

\begin{lemma}\label{lne9}
Let $n=|X|=9$. Then there are $9$-cycles $\al$ and $\bt$ in $\sym(X)$ such that
the distance between $\al$ and $\bt$ in $\cg(\mi(X))$ is $5$.
\end{lemma}
\begin{proof}
Let $X=\{1,2,\ldots,9\}$, and consider the following $9$-cycles in $\sym(X)$:
\[
\al=(1\,2\,3\,4\,5\,8\,7\,6\,9)\,\mbox{ and }\,\bt=(1\,4\,7\,2\,5\,8\,3\,6\,9).
\]
We claim that the distance between $\al$ and $\bt$ in $\cg(\mi(X))$ is $5$. We know that $d(\al,\bt)\leq5$
by Proposition~\ref{pdio}. Suppose to the contrary that $d(\al,\bt)<5$. Then there are $\del,\gam,\eta\in\mi(X)-\{0,1\}$
such that $\al-\del-\gam-\eta-\bt$. Then, by Lemma~\ref{ldia4}, $\del=\al^p$ and $\eta=\bt^q$ for some $p,q\in\{1,\ldots,8\}$.
The exponent $p$ is $3$, $6$, or relatively prime to $9$.
In the latter case, there is $t\in\{1,\ldots,8\}$, relatively prime to $9$, such that $pt\equiv1\!\pmod9$. Since
$\gam$ commutes with $\del=\al^p$, it also commutes with $(\al^p)^{3t}=(\al^{pt})^3=(\al^1)^3=\al^3$. If $p=6$,
then $\gam$ commutes with $(\al^6)^5=(\al^{10})^3=(\al^1)^3=\al^3$. Hence, in either case, $\gam$ commutes with $\al^3$.
By a similar argument, $\gam$ also commutes with $\bt^3$, and so
\[
\al^3=(1\,4\,7)\jo(2\,5\,6)\jo(3\,8\,9)-\gam-(1\,2\,3)\jo(4\,5\,6)\jo(7\,8\,9)=\bt^3.
\]
Since $\gam\ne0$, there is a cycle $\sig$ in $\bt^3$ such that $\spa(\sig)\subseteq\dom(\gam)$ (by Proposition~\ref{pcen}).
Therefore, $1$, $4$, or $7$ is in $\dom(\gam)$, and so, since $\gam$ commutes with $\al^3$ and $(1\,4\,7)$
is a cycle in $\al^3$, we have $(1\,4\,7)\in\dom(h_\gam^{\al^3})$. There are three possible cases.
\vskip 1mm
\noindent{\bf Case 1.} $(1\,4\,7)h_\gam^{\al^3}=(1\,4\,7)$.
\vskip 1mm
Then $1\gam=1$, $4$, or $7$. If $1\gam=1$, then $4\gam=4$ and $7\gam=7$ by Proposition~\ref{pcen}. But then, since $\gam$
commutes with $\bt^3$, $\gam$ must fix every element of every cycle of $\bt^3$, that is, $\gam=1$. This is a contradiction.
Suppose $1\gam=4$. Then $(1\,2\,3)h_\gam^{\bt^3}=(4\,5\,6)$ with $2\gam=5$ and $3\gam=6$. But then
$(2\,5\,6)h_\gam^{\al^3}=(2\,5\,6)$ and $(3\,8\,9)h_\gam^{\al^3}=(2\,5\,6)$, which is a contradiction
since $h_\gam^{\al^3}$ is injective. If $1\gam=7$, we obtain a contradiction in a similar way.
\vskip 1mm
\noindent{\bf Case 2.} $(1\,4\,7)h_\gam^{\al^3}=(2\,5\,6)$.
\vskip 1mm
Then $1\gam=2$, $5$, or $6$. If $1\gam=2$, then $4\gam=5$ and $7\gam=6$, and so $(4\,5\,6)h_\gam^{\bt^3}=(4\,5\,6)$
and $(7\,8\,9)h_\gam^{\bt^3}=(4\,5\,6)$. This is a contradiction since $h_\gam^{\bt^3}$ is injective.
If $1\gam=5$, then $4\gam=6$, and so $(1\,2\,3)h_\gam^{\bt^3}=(4\,5\,6)$ and $(4\,5\,6)h_\gam^{\bt^3}=(4\,5\,6)$,
again a contradiction. Finally, if
$1\gam=6$, then $7\gam=5$, and so $(1\,2\,3)h_\gam^{\bt^3}=(4\,5\,6)$ and $(7\,8\,9)h_\gam^{\bt^3}=(4\,5\,6)$,
also a contradiction. 
\vskip 1mm
\noindent{\bf Case 3.} $(1\,4\,7)h_\gam^{\al^3}=(3\,8\,9)$.
\vskip 1mm
In this case, we also obtain a contradiction by the argument similar to the one used in Case~2.

Therefore, the assumption $d(\al,\bt)<5$ leads to a contradiction, and so $d(\al,\bt)\geq5$.
Since we already know that $d(\al,\bt)\leq5$, we have $d(\al,\bt)=5$.
\end{proof}

\begin{theorem}\label{tpow}
Let $|X|=n=p^k$, where $p$ is an odd prime and $k\ge2$.  Then the diameter of $\cg(\mi(X))$ 
is $5$.
\end{theorem}
\begin{proof}
By Proposition~\ref{pdio}, it suffices to find two $n$-cycles $\al$ and $\bt$ in $\sym(X)$ such that
the distance between $\al$ and $\bt$ in $\cg(\mi(X))$ is at least $5$. If $n=9$, then such cycles exist by Lemma~\ref{lne9}.

Suppose $n>9$ and let $X=\{1,2,\ldots,n\}$. 
If $\al,\bt \in \mi(X)$ are 
$n$-cycles such that  $\al-\del-\gam-\eta-\bt$ for some $\del,\gam,\eta \in \mi(X)-\{0,1\}$,
then, by the argument similar to the one we used in Lemma~\ref{lne9},
we may assume that $\del=\al^q$ and $\eta=\bt^q$, where $q=p^{k-1}$. Note that then $\del$ and $\eta$
are joins of $q$ cycles, each cycle of length $p$. Consider the following $\del,\eta\in\sym(X)$:
\[
\del=(1 \;\; 2 \;\; \ldots \;\;p) \jo (p+1\;\; p+2 \;\; \ldots\;\; 2p)\jo \cdots \jo (n-p+1 \;\; n-p+2\;\; \ldots \;\;n),
\]
\begin{align*}
\eta=&~~~\,(1  &&q-1& &2q -2   &&\ldots  &&n-3q-p+3 &&n-2q-p+2 &&n-q-p+1) \\
 &\jo (2  &&q && 2q-1 && \ldots &&n-3q-p+4 && n-2q-p+3 && n-q-p+2) \\
&\jo  (3   &&q+1  &&2q    &&\ldots && n-3q-p+5  &&n-2q-p+4  &&n-q-p+3) \\
&\jo  (4   &&q+2  &&2q+1  &&\ldots && n-3q-p+6  &&n-2q-p+5  &&n-q-p+4) \\
&~\vdots \\
&\jo  (q-3  &&2q-5  &&3q-6   &&\ldots && n-2q-p-1  &&n-q-p-2  &&n-p-3) \\
&\jo  (q-2  &&2q-4  &&3q-5     &&\ldots && n-2q-p &&n-q-p-1  &&n-p-2) \\
&\jo  (n-p+1  &&2q-3  && 3q-4     &&\ldots && n-2q-p+1  &&n-q-p  &&n-p-1) \\
&\jo  (n-p+2  &&n-p+3  &&n-p+4   &&\ldots && n-1 &&n  &&n-p) \\
\end{align*}
The construction of $\del$ is straightforward. Regarding $\eta$,
the last cycle,
\[
\tau=(n-p+2\,\,\,n-p+3\,\,\,n-p+4\,\,\ldots\,\,\,n-1\,\,\,n\,\,\,n-p),
\]
is special. (Its role will become clear in the second part of the proof). If $\tau'=(x_1\,x_2\ldots\,x_p)$
is any other cycle in $\eta$, then $x_{i+1}-x_i=q-1$ for every $i\in\{2,\ldots,p-1\}$. Here and in the following, we
assume cycles are always represented  by expressions listing the elements in the fixed orders  from the definitions of $\delta$ 
and $\eta$, so that we may speak of the position of an element in a cycle.
 
Let $\al$ and $\bt$ be $n$-cycles such that $\al^q=\del$ and $\bt^q=\eta$. As $\del$ and $\eta$ consist of 
$q$ disjoint cycles of length $p$, such $\al$ and $\bt$ exist. 
We claim that $d(\al,\bt)\geq5$. Suppose to the contrary that $d(\al,\bt)<5$. Then, 
by the foregoing argument, there exists $\gam\in \mi(X)-\{0,1\}$ with $\del-\gam-\eta$. 

Define a binary relation $\sim$ on $X$ by: $x\sim y$ if there exists a cycle $\rho$ in $\del$ or in $\eta$ with 
$\{x,y\} \subseteq \spa(\rho)$. Let $\sim^*$ be the transitive closure of $\sim$.
It follows from  Proposition~\ref{pcen} that $\sim$ preserves the following two properties:
``$\gam$ is defined at  $x$'' and ``$\gam$ fixes $x$''. It is then clear that $\sim^*$ preserves
these properties as well. We will write $x\sim_\del y$ if $x$ and $y$ are in the same cycle of $\del$, and
$x\sim_\eta y$ if $x$ and $y$ are in the same cycle of $\eta$ (so $\sim\,=\,\sim_\del\cup\sim_\eta$).

We claim that $\sim^*\,=X\times X$. Consider the set $A=\{n-q-p+1,n-q-p+2,\ldots,n-p\}$ of the rightmost
elements of the cycles in $\eta$. Note that $A$ contains $t=q/p$ multiples of $p$:
\begin{equation}\label{e1tpow}
n-p,\,\,n-2p,\,\ldots,\,\,n-tp.
\end{equation}
Let $i\in\{1,2,\ldots,t-1\}$. We claim that $n-ip\sim^*n-(i+1)p$. First, we have $n-ip\sim_\del n-ip-p+1$
since $(n-ip-p+1\,\,n-ip-p+2\,\ldots\,\,n-ip)$ is a cycle in $\del$. Next, $n-ip-p+1$ is a rightmost element of a cycle
in $\eta$ that is different from $\tau$ (the last cycle). We have already observed that $n-ip-p+1-(q-1)$ is the 
preceding element in the same cycle. Thus
\[
n-ip-p+1\sim_\eta n-ip-p+1-(q-1)=n-q-ip-p+2.
\]
Further, $n-q-ip-p+2\sim_\del n-q-ip-p+1$, and finally
\[
n-q-ip-p+1\sim_\eta n-q-ip-p+1+(q-1)=n-ip-p=n-(i+1)p.
\]
To summarize,
\[
n-ip \sim_\del  n-ip-p+1\sim_\eta  n-q-ip-p+2 \sim_\del  n-q-ip-p+1 \sim_\eta  n-ip-p= n-(i+1)p.
\]
It follows by the transitivity of $\sim^*$ that any two multiples of $p$ from (\ref{e1tpow}) are $\sim^*$-related.
Let $x,y\in X$. Then there are $z,w\in A$ such that $x\sim_\eta z$ and $y\sim_\eta w$. Now, $z$ must be in some
cycle of $\del$ whose rightmost element is a multiple of $p$. Since $z\in A$, that multiple must come
from $(\ref{e1tpow})$, that is, $z\sim_\del n-jp$ for some $j\in\{1,2,\ldots,t\}$, where $t=q/p$.
Similarly, $w\sim_\del n-lp$ for some $l\in\{1,2,\ldots,t\}$. Hence
\[
x\sim_\eta z\sim_\del n-jp\sim^*n-lp\sim_\del w\sim_\eta y.
\]
Thus $x\sim^* y$, and so $\sim^*\,=X\times X$.

As $\gam \ne 0$, $\gam$ must be defined on some element of $X$. Since $\sim^*$ preserves the statement
``$\gam$ is defined at $x$'' and $\sim^*\,=X\times X$, we have $\dom(\gam)=X$.

Consider the cycle $\rho=( n-p+1 \;\; n-p+2\;\; \ldots \;\;n)$ in $\del$ and the cycle
\[
\tau=(n-p+2  \;\; n-p+3  \;\; n-p+4   \;\; \ldots \;\;  n-1  \;\; n   \;\; n-p)
\]
in $\eta$, and note
that $\spa(\rho)\cap\spa(\tau)$ consists of $p-1$ elements. Thus,
$\spa(\rho h_\gam^\del)\cap\spa(\tau h_\gam^\eta)$ also consists of $p-1$ elements.
However, for all cycles $\rho'$ in $\del$ and $\tau'$ in $\eta$, if $\rho'\ne\rho$, then
$\spa(\rho')\cap\spa(\tau')$ consists of either $1$ or $2$ elements, where $2$ 
is only possible when $n=p^2$. In the latter case, $p\geq5$ (since $n>9$), and so $p-1>2$. 
If $n=p^k$ with $k>2$, then $p-1>1$ (even when $p=3$). Hence $\rho h_\gam^\del=\rho$ since otherwise
we would have $|\spa(\rho h_\gam^\del)\cap\spa(\tau h_\gam^\eta)|<p-1$. 
Applying the same argument to $\tau$, we see that $\tau h_\gam^\eta= \tau$. 

These two conditions imply that the element $n-p$ that occurs in $\tau$ must be fixed by $\gam$.
Since $\sim^*$ preserves the statement
``$\gam$ fixes $x$'' and $\sim^*\,=X\times X$, it follows that $\gam$ fixes every element of $X$.
So $\gam=1$, which is a contradiction. We have proved that $d(\al,\bt)\geq5$.

It now follows from Proposition~\ref{pdio} that the diameter of $\cg(\mi(X))$ is $5$.
\end{proof} 

The problem of finding the exact value of the diameter of $\cg(\mi(X))$ when $n$ is odd and divisible
by at least two primes remains open.
By Lemmas~\ref{ldia2} and~\ref{ldia3}, $d(\al,\bt)\leq4$ for all $\al,\bt\in\mi(X)$ such that $\al$ or $\bt$ is not an $n$-cycle.
So the exact value of the diameter (which is $4$ or $5$) depends on the answer to the following question.
\vskip 1mm
\noindent{\bf Question.} Let $n\geq 15$ be odd and divisible by at least two primes. Are there $n$-cycles $\al,\bt\in\mi(X)$
such that $d(\al,\bt)=5$?

We conclude this section with a discussion of the diameter of the commuting graph
of the symmetric group $\sym(X)$. 
Iranmanesh and Jafarzadeh have proved \cite[Theorem~3.1]{IrJa08} that if $n$ and $n-1$ are not primes,
then the diameter of $\cg(\sym(X))$ is at most $5$. (If $n$ or $n-1$ is a prime, then the diameter
of $\cg(\sym(X))$ is $\infty$.)

Dol\u{z}an and Oblak have strengthened this result \cite[Theorem~4]{DoOb11} by showing that if $n$ and $n-1$
are not primes, then the distance between $\al=(1\,2\,\ldots\,n)$ and $\bt=(1\,2\, \ldots\, n-1)\jo(n)$
in $\cg(\sym(X))$ is at least $5$ (so the diameter of $\cg(\sym(X))$ is exactly $5$).
However, their proof contains a gap.
They state that if $\rho,\sig,\tau\in\sym(X)$ are such that $\rho-\sig-\tau$ and the length of any cycle in
$\rho$ is relatively prime to the length of any cycle in $\tau$, then $\sig$ must fix every point in $X$,
and so $\sig=1$. However, this statement is not true, even with the additional assumptions that $\rho$ is the power of 
an $n$-cycle and $\tau$ is the power of a disjoint join between an $(n-1)$-cycle and a $1$-cycle.
Let $X=\{1,2,\ldots,10\}$, and consider
\[
\rho=(1\,2)\jo(3\,4)\jo(5\,6)\jo(7\,8)\jo(9\,10)\,\,\mbox{ and }\,\,\tau=(1\,3\,5)\jo(2\,4\,6)\jo(7\,8\,9)\jo(10).
\]
Then for $\sig=(1\,3\,5)\jo(2\,4\,6)\jo(7)\jo(8)\jo(9)\jo(10)$, we have $\rho-\sig-\tau$ but $\sig\ne1$. 

It is possible to fix this gap by taking into account the special form of $\al$ and $\bt$ in the
original proof. We do this in the following lemma.

\begin{lemma}\label{ldol}
Let $X=\{1,2,\ldots,n\}$, where neither $n$ nor $n-1$ is a prime. Then, the distance between $\al=(1\,2\,\ldots\,n)$ and 
$\bt=(1\,2\, \ldots\, n-1)\jo(n)$ in $\cg(\sym(X))$ is at least $5$.
\end{lemma}

\begin{proof}
Suppose to the contrary that $d(\al,\bt)<5$. Then $\al-\rho-\sig-\tau-\bt$ for some $\rho,\sig,\tau\in\sym(X)-\{1\}$.
It easily follows from the proof of Lemma~\ref{ldia4} that $\rho=\al^m$ and $\tau=\bt^k$
for some $m\in\{1,\ldots,n-1\}$ and some $k\in\{1,\ldots,n-2\}$. We may assume that $m$ is a proper divisor of $n$.
(If $m$ and $n$ are relatively prime, then $\al^m=\al$, and so we may replace $\rho=\al^m$ in $\al-\rho-\sig-\tau-\bt$
with $\al^{m'}$, where $m'$ is any proper divisor of $n$. Similarly, if $m=lm'$, where $l$ and $n$ are relatively prime
and $m'$ is a proper divisor of $n$, we can replace $\rho=\al^m$ with $\al^{m'}$.) Similarly,
we may assume that $k$ is a proper divisor of $n-1$. Note that $m$ and $k$ are relatively prime.
The permutation $\rho=\al^m$ is the join of $m$ cycles, each of length $t=n/m$:
\begin{equation}\label{e1ldol}
\rho=\al^m=\lam_1\jo\lam_2\jo\cdots\jo\lam_m.
\end{equation}
Consider the cyclic group $\mathbb Z_n=\{1,2,\ldots,n\}$ of integers modulo $n$ and the subgroup $\langle m\rangle$
of $\mathbb Z_n$. Then the spans of the cycles in $\rho=\al^m$ are precisely the cosets of the group $\langle m\rangle$.
Since $k$ and $m$ are relatively prime, the cosets of $\langle m\rangle$ are
\[
\langle m\rangle+k,\,\,\langle m\rangle+2k,\,\ldots,\,\,\langle m\rangle+mk.
\]
We may order the cycles in (\ref{e1ldol}) in such a way that $\spa(\lam_i)=\langle m\rangle+ik$ for every $i\in\{1,2,\ldots,m\}$.

Since $(n)$ is the only $1$-cycle in $\tau=\bt^k$, $\sig$ fixes $n$ by Proposition~\ref{pcen}. 
Recall that, by Proposition~\ref{pcen}, if $\sig$ fixes some element of a cycle in $\rho$ or in $\tau$,
then it fixes all elements of that cycle.
Thus $\sig$ fixes all elements of $\spa(\lam_m)$ (since $\spa(\lam_m)=\langle m\rangle+mk=\langle m\rangle$ contains $n$).
Since $m \le n/2$ and $k \le (n-1)/2$, there is $x\in\spa(\lam_m)$ such that 
$x+k\le n-1$ (in standard, non-modular addition). Thus $x$ and $x+k$ are in the same cycle of $\tau$ (since
$\tau=\bt^k$ is a join of $(n)$ and $k$ cycles, each of length $s=(n-1)/k$, and the span of each cycle of length $s$
is closed under addition of $k$ modulo $n-1$). Hence, since $\sig$ fixes $x$, $\sig$ also fixes $x+k$.
But $x+k\in\spa(\lam_1)$ (since $\spa(\lam_1)=\langle m\rangle+k=(\langle m\rangle+mk)+k$), and so $\sig$ fixes
all elements of $\spa(\lam_1)$.

Applying the foregoing argument $m-2$ more times, to cycles $\lam_1,\ldots,\lam_{m-1}$, will show
that $\sig$ fixes all elements of every cycle in $\rho$. Hence $\sig=1$, which is a contradiction.
Thus $d(\al,\bt)\geq5$.
\end{proof}

\section{Problems}\label{spro}

In the process of proving Theorem~\ref{tgen},
we came across a purely combinatorial 
conjecture that, if true, could simplify some of the proofs. 
As this combinatorial problem may be of interest regardless of the commuting graphs, we present it here. 

\begin{prob}
{\rm
Let $s,t>1$ be natural numbers. Suppose $A$ is an $s\times t$ matrix with entries from some set $S$
such that:
\begin{itemize}
  \item [(a)] entries in each row of $A$ are pairwise distinct;
  \item [(b)] entries in each column of $A$ are pairwise distinct; and
  \item [(c)] there is no $a\in S$ such that $a$ occurs in \emph{every} row of $A$ or $a$ occurs in \emph{every} column of $A$.
\end{itemize}
For given $s$ and $t$ find the smallest $S$ that satisfies the three conditions above.  
In particular, is it necessarily true that $A$ contains at least $s+t-1$ distinct entries?
}
\end{prob}

For a graph $G=(V,E)$, denote by $\aut(G)$ the group of automorphisms of $G$. Recall that $T(X)$ denotes
the semigroup of full transformations on $X$.
The automorphism groups of the commuting graphs of $T(X)$ and of $\mi(X)$ are, 
comparatively to the size of the graphs themselves, very large. 
We list here their cardinalities for small values of $n=|X|$, which we have obtained using GAP \cite{Scel92} and GRAPE \cite{So06}.
\[
\begin{tabular}{|c|c|c|}\hline
$n$&$|\aut(\cg(\mi(X)))|$&$|\aut(\cg(T(X)))|$\\ \hline
$2$&$2^{2}\cdot 3$&$2\cdot3$\\ \hline
$3$&$2^{9}\cdot3$&$2^{5}\cdot 3^{4}$\\ \hline
$4$&$2^{38}\cdot3^5$&$2^{34}\cdot3$\\ \hline
$5$&$2^{231}\cdot 3^{44}\cdot 5^2$&$2^{410}\cdot3^9\cdot5^2$\\\hline
\end{tabular}
\]

\begin{prob}
{\rm
Describe the automorphism groups of the commuting graphs of $\mi(X)$, $T(X)$, and $Sym(X)$.
} 
\end{prob}

The diameter of the commuting graph of $T(X)$ has been determined in \cite[Theorems~2.22]{ArKiKo11}.

\begin{prob}
{\rm
Find the clique number of the commuting graph of $T(X)$.
}
\end{prob}

A related problem is to determine the chromatic number of a given commuting graph.

\begin{prob}
{\rm
Find the chromatic numbers of the commuting graphs of $\mi(X)$, $T(X)$, and $\sym(X)$.  
}
\end{prob}

It has been proved in \cite[Theorem~4.1]{ArKiKo11}
that for every natural $n$, there exists a 
semigroup (consisting of idempotents) such that
the diameter of its commuting graph is $n$. 
It has been conjectured
that there exists a common upper bound of the diameters of the (connected) commuting graphs of finite groups. 

\begin{prob}
{\rm 
Is it true that for every natural $n$, there  exists a finite \emph{inverse} semigroup whose commuting graph has diameter $n$? 
}
\end{prob}

The commuting graphs of finite groups have attracted a great deal of attention. There is a parallel
concept of the non-commuting graph of a finite group, which has also been the object of intensive study
\cite{AbAkMa06,Da09,Ne76,ZhSh05}. (A \emph{non-commuting graph} of a finite nonabelian group $G$ 
is a simple graph whose
vertices are the non-central elements of $G$ and two distinct vertices $x,y$ are adjacent if $xy\ne yx$.)
Once again, the concept carries over to semigroups, but nothing is known about the non-commuting graphs of semigroups. 

\begin{prob}
{\rm 
Find the diameters, clique numbers, and chromatic numbers of the non-commuting 
graphs of $T(X)$ and $\mi(X)$.  Is it true that for every natural $n$, there exists a 
semigroup whose non-commuting graph has diameter $n$? 
}
\end{prob}

In the present paper and \cite{ArKiKo11}, the commuting graphs of $\mi(X)$, $T(X)$, and their ideals have been
investigated. However, there are many other subsemigroups of $\mi(X)$ and $T(X)$ that have been intensively
studied (see \cite{vhf,GaMa09}).

\begin{prob}
{\rm 
Calculate the diameters, clique numbers, and chromatic numbers
of commuting and non-commuting graphs of various subsemigroups of $\mi(X)$ and $T(X)$.
}
\end{prob}


\noindent {\bf Acknowledgment}
The first author was partially supported  by FCT through the following projects: PEst-OE/MAT/UI1043/2011, Strategic Project of Centro de \'Algebra da Universidade de Lisboa; and PTDC/MAT/101993/2008, Project Computations in groups and semigroups .

The research of the second author leading to these results has received funding from the
European Union Seventh Framework Programme (FP7/2007-2013) under
grant agreement no.\ PCOFUND-GA-2009-246542 and from the Foundation for 
Science and Technology of Portugal.

\end{document}